\documentclass[final,1p,times]{elsarticle}

\usepackage[colorlinks=true]{hyperref}
\usepackage{mathrsfs}
\usepackage{mathrsfs,curves}
\usepackage{amsthm,amsmath,amssymb,mathrsfs,amsfonts,graphicx,graphics,latexsym,exscale,cmmib57,dsfont,bbm,amscd,ulem}
\usepackage{color,soul}
\newtheorem{thm}{Theorem}[section]
\newtheorem{cor}[thm]{Corollary}
\newtheorem{lem}[thm]{Lemma}
\newtheorem{Q}[thm]{Question}
\newtheorem{prop}[thm]{Proposition}

\theoremstyle{definition}
\newtheorem{BF}{Basic Fact}
\newtheorem{defn}[thm]{Definition}

\newtheorem{note}{Note}
\newtheorem{rem}[thm]{Remark}

\newtheorem{exa}[thm]{Example}

\theoremstyle{remark}

\journal{JDE April 1, 2017 (accepted June 19, 2017)}

\begin{document}

\begin{frontmatter}

\title{Devaney chaos, Li-Yorke chaos, and multi-dimensional Li-Yorke chaos for topological dynamics}
\author{Xiongping Dai}
\ead{xpdai@nju.edu.cn}
\author{Xinjia Tang}
\address{Department of Mathematics, Nanjing University, Nanjing 210093, People's Republic of China}

\begin{abstract}
Let $\pi\colon T\times X\rightarrow X$, written $T\curvearrowright_\pi X$, be a  topological semiflow/flow on a uniform space $X$ with $T$ a multiplicative topological semigroup/group not necessarily discrete. We then prove:
\begin{itemize}
\item If $T\curvearrowright_\pi X$ is non-minimal topologically transitive with dense almost periodic points, then it is sensitive to initial conditions. As a result of this, Devaney chaos $\Rightarrow$ Sensitivity to initial conditions, for this very general setting.
\end{itemize}
Let $\mathbb{R}_+\curvearrowright_\pi X$ be a $C^0$-semiflow on a Polish space; then we show:
\begin{itemize}
\item If $\mathbb{R}_+\curvearrowright_\pi X$ is topologically transitive with at least one periodic point $p$ and there is a dense orbit with no nonempty interior, then it is \textit{multi-dimensional Li-Yorke chaotic}; that is, there is a uncountable set $\Theta\subseteq X$ such that for any $k\ge2$ and any distinct points $x_1,\dotsc,x_k\in\Theta$, one can find two time sequences $s_n\to\infty, t_n\to\infty$ with
    \begin{equation*}
    s_n(x_1,\dotsc,x_k)\to(x_1,\dotsc,x_k)\in X^k\quad \textrm{and}\quad t_n(x_1,\dotsc,x_k)\to(p,\dotsc,p)\in\varDelta_{X^k}.
    \end{equation*}
    Consequently, Devaney chaos $\Rightarrow$ Multi-dimensional Li-Yorke chaos.
\end{itemize}
Moreover, let $X$ be a non-singleton Polish space; then we prove:
\begin{itemize}
\item Any weakly-mixing $C^0$-semiflow $\mathbb{R}_+\curvearrowright_\pi X$ is densely multi-dimensional Li-Yorke chaotic.
\item Any minimal weakly-mixing topological flow $T\curvearrowright_\pi X$ with $T$ abelian is densely multi-dimensional Li-Yorke chaotic.
\item Any weakly-mixing topological flow $T\curvearrowright_\pi X$ is densely Li-Yorke chaotic.
\end{itemize}
We in addition construct a completely Li-Yorke chaotic minimal $\textrm{SL}(2,\mathbb{R})$-acting flow on the compact metric space $\mathbb{R}\cup\{\infty\}$. Our various chaotic dynamics are sensitive to the choices of the topology of the phase semigroup/group $T$.
\end{abstract}

\end{frontmatter}

\section{Introduction}\label{sec1}
We begin by briefly recalling the definition of topological dynamical systems. Let $T$ be an arbitrary multiplicative topological semigroup not necessarily discrete, which jointly continuously acts from left on a topological space $X$ via an acting map $\pi\colon(t,x)\mapsto\pi(t,x)=\pi_t(x)$, written as
\begin{gather*}
\pi\colon T\times X\rightarrow X\ \textrm{ or }\ T\curvearrowright_\pi X,
\end{gather*}
such that
\begin{itemize}
\item $\pi_s\pi_t(x)=\pi_{st}(x)\ \forall s,t\in T, x\in X$ and
\item $\pi_e=i_X\colon x\mapsto x$ is the identity map of $X$ onto itself if $e$ is the neutral element of $T$ whenever $T$ is a monoid.
\end{itemize}
For simplicity, sometimes we will identify the transition map $\pi_t\colon x\mapsto \pi(t,x)$ with $t\colon x\mapsto tx$, for each $t\in T$, if no confusion arises.

Here the triple $(T,X,\pi)$, or simply $(T,X)$ and $T\curvearrowright_\pi X$, will be called a \textit{topological semiflow} with phase space $X$ and with phase semigroup $T$. If $T$ is just a topological group, then $(T,X,\pi)$ or $T\curvearrowright_\pi X$ will be called a \textit{topological flow} with phase group $T$.

There are systematical studies on chaos for $\mathbb{Z}$- or $\mathbb{Z}_+$-action dynamical systems on compact metric spaces; see, e.g., \cite{HY, AGHSY, LY} and references therein. In this paper, we will study Devaney chaos, Li-Yorke chaos, and multi-dimensional Li-Yorke chaos of flows or semiflows $(T,X,\pi)$ on uniform spaces $X$ with general topological phase group or semigroup $T$.

First of all, it should be noticed that chaotic dynamics are very different for general group or semigroup actions, even for $\mathbb{R}$ and $\mathbb{R}_+$, with the well-known $\mathbb{Z}$- or $\mathbb{Z}_+$-actions. Let's see an example, which is already beyond the setting of the following important chaos criterion:
\begin{itemize}
\item Let $f\colon X\rightarrow X$ be a topologically transitive continuous self-map of a compact metric space $X$ with no isolated points. If there is some subsystem $(Y,f)$ of $(X,f)$ such that $(X\times Y,f)$ is topologically transitive, then $(X,f)$ is densely uniformly (Li-Yorke) chaotic. See \cite[Theorem~3.1]{AGHSY}.
\end{itemize}

\begin{exa}\label{e1}
Let $X=\mathbb{R}\cup\{\infty\}$ which is the one-point compactification of $\mathbb{R}$ homeomorphic to the unit circle $\mathbb{T}$ with no isolated points, and let $T=(\mathbb{R},+)$ with the usual topology. Define the flow $\pi\colon (t,x)\mapsto x+t$ of $T\times X$ onto $X$. Clearly, $(T,X,\pi)$ is a non-minimal topologically transitive with a unique fixed point $\infty$. However, this topological flow has no Li-Yorke chaotic pairs.
\end{exa}

In Example~\ref{e1}, if let $Y=\{\infty\}$, then $(T,X\times Y)$ is topologically transitive. So while all of the conditions of the above criterion of Akin et al. are satisfied, the topological flow $\mathbb{R}\curvearrowright_\pi X$ is not at all Li-Yorke chaotic as is shown by the following figure.

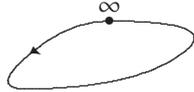
\setlength{\unitlength}{1mm}
\begin{picture}(50, 17)
\put(25,12){\circle*{1}}
\put(23.5,13){$\infty$}
\closecurve(15,8,14,3,35,8,25,12)
\put(15,8){\vector(2,-1){1}}
\end{picture}

In addition, for a discrete-time topological semiflow $\mathbb{Z}_+\curvearrowright_\pi X$ on a metric space $(X,d)$, we say $x,y\in X, x\not=y$ is a Li-Yorke chaotic pair if there are two infinite sequences $\{s_n\}, \{t_n\}$ in $\mathbb{Z}_+$ with
\begin{equation*}
\lim_{n\to\infty}d(t_nx,t_ny)=0\quad \textrm{and}\quad \lim_{n\to\infty}d(s_nx,s_ny)>0.
\end{equation*}
However, for any continuous-time topological semiflow $\mathbb{R}_+\curvearrowright_\pi X$, one can always take an infinite time-sequence $s_n\downarrow0$ so that
\begin{equation*}
\lim_{n\to\infty}d(s_nx,s_ny)=d(x,y)>0.
\end{equation*}
This shows that the definition of Li-Yorke chaos for any topological semiflow/flow with general phase semigroup/group should be more subtler than the classical $\mathbb{Z}$- or $\mathbb{Z}_+$-acting case.

Moreover, it should be mentioned that the topology of the phase semigroup/group $T$ will play an important role in our later discussion; this is because different topology of $T$ induces different periodic points, almost periodic points and recurrent points included in our chaotic dynamics. For instance, let
\begin{equation*}
\pi\colon\mathbb{R}\times\mathbb{R}/\mathbb{Z}\rightarrow\mathbb{R}/\mathbb{Z};\quad \textrm{by }(t,x)\mapsto t+x \pmod{1}
\end{equation*}
which is such that $\{t\,|\,\pi_t(x)=x\}=\mathbb{Z}$. Then $(\mathbb{R},\mathbb{R}/\mathbb{Z},\pi)$ is pointwise periodic, for $\mathbb{Z}$ is syndetic in $\mathbb{R}$, under the usual topology of $\mathbb{R}$; however, it has no periodic points, for $\mathbb{Z}$ is not syndetic in $\mathbb{R}$, under the discrete topology.

\subsection{Sensitivity of semiflows}\label{sec1.1}
Let $\mathscr{U}_X$ be a compatible symmetric uniform structure of a uniform topological space $X$ (cf.~\cite[Appendix~II]{Aus}); then we will write $\varepsilon[x]=\{y\in X\,|\,(x,y)\in\varepsilon\}$, for $x\in X$ and $\varepsilon\in\mathscr{U}_X$.

It is well known that sensitive dependence on initial conditions is one of the basic ideas in chaotic dynamics; see, for example, \cite{LY, AY, Dev, BB, GW, AAB, HY, AK, GM, KM, WLF, SKBS, HLY, YZ, WH, D1, D2, LiY} and so on.
First of all we need to recall the notion of sensitivity in the general setting.
\begin{itemize}
\item A topological semiflow $(T,X,\pi)$ on a uniform space $(X,\mathscr{U}_X)$ with phase semigroup $T$ is said to be \textbf{\textit{sensitive}} to initial conditions in case there exists an $\varepsilon\in\mathscr{U}_X$ (sensitivity index) such that for every $x\in X$ and any $\delta\in\mathscr{U}_X$, one can find some $y\in\delta[x]$ and $t\in T$ such that
    $(tx,ty)\not\in\varepsilon$. See, e.g., \cite{GM,KM,Pol,WLF,Ryb,CC,SKBS,Ar}.
    \begin{itemize}
    \item It is easy to see that if $(T,X,\pi)$ is not sensitive to initial conditions, then $(T,X,\pi)$ is equicontinuous at some point $x\in X$; and conversely, if $(T,X,\pi)$ is equicontinuous at some point of $X$, then it is not sensitive to initial conditions.
    \end{itemize}
\end{itemize}

In our semigroup setting, we will prove the following proposition, which appears as \cite[Main result]{KM} in the special case that $T$ is assumed a C-semigroup (i.e. $T\setminus{Tt_0}$ is relatively compact in the topological semigroup $T$ for each $t_0\in T$) acting on a compact metric space.
\begin{itemize}
\item \textit{Let $(T,X,\pi)$ be a topological semiflow on a uniform space $X$, which is topologically transitive with dense almost periodic points. If $(T,X,\pi)$ is
\begin{itemize}
\item[\textbf{--}] either non-equicontinuous with $T$ abelian and with $X$ a compact Hausdorff space
\item[\textbf{--}] or non-minimal,
\end{itemize}
then $(T,X,\pi)$ is sensitive to initial conditions} (see Proposition~\ref{prop2.8} in $\S\ref{sec2.1}$).
\end{itemize}
Thus, for any topological semiflow on a uniform space, by the second part of the above proposition, it follows that
\begin{itemize}
\item ``Devaney chaos'' (topologically transitive $+$ dense of periodic points $+$ non-minimal; cf.~Definition~\ref{def2.9} in $\S\ref{sec2.2}$) $\Rightarrow$ Sensitivity to initial conditions.
\end{itemize}

This therefore generalizes the available main results in literature such as \cite[Theorem~1.2]{CC} and \cite[Theorem~4.7]{SKBS} which are only for discrete group $T$. Moreover, because the topology of $T$ is not necessarily to be discrete here, our result is useful for the classical $C^0$-semiflows on manifolds with $T=(\mathbb{R},+)$.

In addition, it should be mentioned that our approaches are completely different with those introduced in \cite{BB,HY,KM,CC,SKBS}.

It turns out that the above sensitivity result is a useful tool for our later study of Li-Yorke chaos of weakly-mixing topological semiflows in $\S\ref{sec3}$.

\subsection{Li-Yorke chaos of semiflows}\label{sec1.2}
It is a well-known fact that every topologically weakly-mixing continuous self-map $f$ of a non-trivial compact metric space $X$ is Li-Yorke chaotic (cf.~\cite{XY,HY,AK}). In fact, it is also the case for a topological flow $\pi\colon T\times X\rightarrow X$ on a Polish space with no isolated points. This is just the another main result of this paper.

Similar to the classical Li-Yorke chaos case over a compact metric space, motivated by \cite{KL} for countable discrete groups we now introduce the following notions of chaos we will consider later on.

\begin{defn}\label{def1.2}
Let $(T,X,\pi)$ be a topological semiflow on a metric space $(X,d)$ and $\mathcal{F}=\{F_n\}_{n=1}^\infty$ an increasing sequence of compact subsets of $T$. Such an $\mathcal{F}$ will be referred to as a \textbf{\textit{reference sequence}} for $(T,X,\pi)$.
\begin{itemize}
\item A subset $S$ of $X$ is called a \textbf{\textit{Li-Yorke chaotic set for $(T,X,\pi)$ relative to $\mathcal{F}$}}, provided that
$S$ is a uncountable set and for any $x,y\in S$ with $x\not=y$,
\begin{itemize}
\item $(x,y)$ is a \textit{proximal} pair (i.e. $\lim_{n\to\infty}d(t_nx,t_ny)=0$ for some sequence $\{t_n\}$ in $T$) and
\item $(x,y)$ is a \textit{separated} pair (i.e. $\liminf_{n\to\infty}d(s_nx,s_ny)>0$ for some sequence $\{s_n\}$ in $T$ with $s_n\in T\setminus F_n, n=1,2,\dotsc$).
\end{itemize}

\item $(T,X,\pi)$ is said to be \textbf{\textit{Li-Yorke chaotic}} if there is a Li-Yorke chaotic set $S_\mathcal{F}$ for any reference sequence $\mathcal{F}$ in $T$.
\item $(T,X,\pi)$ is called \textbf{\textit{completely Li-Yorke chaotic}} if $X$ itself is a Li-Yorke chaotic set for $(T,X,\pi)$ relative to any reference sequence $\mathcal{F}$ in $T$.

\item In addition, $(T,X,\pi)$ is called \textbf{\textit{densely Li-Yorke chaotic}} if there is a dense Li-Yorke chaotic set $S_\mathcal{F}$ relative to any reference sequence $\mathcal{F}$ in $T$.
\end{itemize}
\end{defn}

Of course, Li-Yorke chaos itself may be defined for any topological semiflow on any uniform space. However, we are only concerned with flows or semiflows on Polish spaces here.
Note that in our abstract topological semigroup/group acting case, the $\mathcal{F}$-constraint in Definition~\ref{def1.2} is natural and important. For example, let's consider a classical $C^0$-semiflow
\begin{gather*}
\pi\colon T\times X\rightarrow X\quad \textrm{where }T=(\mathbb{R},+).
\end{gather*}
If we only require ``$\exists s_n\in T$ with $\liminf_nd(s_nx,s_ny)>0$'' as \cite[Definition~4.3]{WLF} with no the $F_n$-constraint $s_n\not\in F_n$, then $(x,y)$ may not be a real Li-Yorke chaotic pair for $\mathbb{R}\curvearrowright_\pi X$ such as if $s_n\to0$ then $d(s_nx,s_ny)\to d(x,y)>0$ but $(x,y)$ is possibly an asymptotically proximal pair mentioned before. However, if we take the reference sequence $F_n=[-n,n]$ for $n=1,2,\dotsc$, then we can just avoid this awkward position.

We shall prove the following statement in our very general setting in this paper.
\begin{itemize}
\item \textit{Devaney chaos $\Rightarrow$ Li-Yorke chaos} (see Corollary~\ref{cor2.16}, Proposition~\ref{prop2.17}, and Proposition~\ref{prop2.21} in $\S\ref{sec2.3}$).
\end{itemize}

Notice that our Li-Yorke chaos is different with the classical Li-Yorke case~\cite{LY}, because here $(T,X,\pi)$ does not need to have any periodic or fixed points. Moreover, our Li-Yorke chaos is also different with Devaney chaos (\cite[Def.~3.12]{SKBS} and Def.~\ref{def2.9} in $\S\ref{sec2.2}$ below) that requires $(T,X,\pi)$ is non-minimal and topologically transitive with dense periodic points. In fact, Example~\ref{e1.3} below is a Li-Yorke chaotic weakly-mixing minimal topological flow on a compact metric space.

M.~Misiurewicz proved in 1985~\cite{Mis} that there exists a continuous map of the unit interval onto itself for which there exists a chaotic set of Lebesgue measure $1$. Moreover, in 1999~\cite{Mai} J.~Mai constructed an example of $X$ which admits a completely Li-Yorke chaotic continuous map. Furthermore, Huang and Ye in 2001 \cite{HY01} showed that there are compacta admitting completely Li-Yorke chaotic homeomorphisms.

Nevertheless it should be noted here that \textit{there exists no completely Li-Yorke chaotic minimal semiflow on a compact Hausdorff space with an abelian acting semigroup} (cf.~\cite[Proposition~2.8]{Dai}). This implies that each minimal subsystem of Mai's and Huang and Ye's completely Li-Yorke chaotic examples just consists of one point.

We will, however, begin our discussion on Li-Yorke chaos with introducing a completely Li-Yorke chaotic \textit{minimal topological flow} on a compact metric space with a \textit{non-abelian} canonical phase group.

\begin{exa}\label{e1.3}
Let $X=\mathbb{R}\cup\{\infty\}$, which is homeomorphic to the unit circle $\mathbb{T}$, and $T=\textrm{SL}(2,\mathbb{R})$ (the locally compact second countable topological group of $2\times 2$ real matrices with determinant $1$). Our action of $T$ on $X$ is defined by
\begin{gather*}
\pi\colon T\times X\rightarrow X;\quad (t,x)\mapsto tx=\frac{ax+b}{cx+d}\quad \forall t=\left[\begin{matrix}a&b\\c&d\end{matrix}\right]\in T\textrm{ and }x\in X,
\end{gather*}
such that $t\infty=\frac{a}{c}$ for any $t=\left[\begin{smallmatrix}a&b\\c&d\end{smallmatrix}\right]\in T$.
It is well known and easily checked that the action of $T$ on $X$ is such that if $x,x^\prime\in X$ there is a $t\in T$ with $tx=x^\prime$. So $(T,X,\pi)$ is minimal and pointwise periodic (i.e. $\{t\in T\,|\,tx=x\}$ is syndetic in $T$ for each $x\in X$; cf.~\cite[Exercise~1.3]{Aus}). Moreover, if $x$ and $x^\prime$ are distinct points of $X$, then there is a $t\in T$ with $tx=0$ and $tx^\prime=1$ (see \cite[p.~29]{Aus}). Thus, to show that $(T,X,\pi)$ is completely Li-Yorke chaotic, it is sufficient to show that $(0,1)\in X\times X$ is a Li-Yorke chaotic pair, i.e., it is proximal and separated for $(T,X,\pi)$. For this, let $\{F_n\}_{n=1}^\infty$ be any reference sequence in $T$ and let
\begin{gather*}
t_n=\left[\begin{matrix}n^{-1}&0\\1&n\end{matrix}\right]\quad \textrm{and}\quad s_n=\left[\begin{matrix}n&0\\1&n^{-1}\end{matrix}\right]\quad n=1,2,\dotsc
\end{gather*}
and note that $s_n\not\in F_n$ by dropping some $n$ if necessary and
\begin{gather*}
\lim_{n\to\infty}t_n0=\lim_{n\to\infty}t_n1=0\quad \textrm{and}\quad 0=\lim_{n\to\infty}s_n0\not=\lim_{n\to\infty}s_n1=\infty;
\end{gather*}
therefore each $(x,x^\prime)\in X\times X$ with $x\not=x^\prime$ is a Li-Yorke chaotic pair for $(T,X,\pi)$.
This completes our construction of a \textit{completely Li-Yorke chaotic minimal flow on a compact metric space with non-abelian phase group}.

Moreover, it is easy to verify that $(T,X,\pi)$ admits no invariant quasi-regular Borel probability measure on $X$, and so $\textrm{SL}(2,\mathbb{R})$ is not amenable. \hfill$\square$
\end{exa}


Next based on R.~Ellis' ``two-circle'' minimal set, we will construct another example with non-metrizable phase space, which is not Li-Yorke chaotic but which contains infinitely many Li-Yorke chaotic pairs and which is sensitive to initial conditions.

\begin{exa}\label{e1.4}
Let $\mathbb{T}$ be the circle, regarded as the real numbers modulo $1$, and let $Y_0=\mathbb{T}\times\{0\}$ and $Y_1=\mathbb{T}\times\{1\}$ be two distinct copies of $\mathbb{T}$. Let $T=\mathbb{T}$ with the discrete topology and $X=Y_0\cup Y_1$. Here $X$ will be topologized by specifying an open-closed neighborhood base for each point as follows. If $\varepsilon>0$, let
\begin{align*}
N_\varepsilon(y,0)&=\{(y+r,0)\,|\,0\le r<\varepsilon\}\cup\{(y+r,1)\,|\,0<r\le\varepsilon\}\\
\intertext{and}
N_\varepsilon(y,1)&=\{(y+r,0)\,|\,0>r\ge-\varepsilon\}\cup\{(y+r,1)\,|\,0\ge r>-\varepsilon\};
\end{align*}
then $N_\varepsilon(y,0)$ and $N_\varepsilon(y,1)$ are open-closed basic neighborhoods of $(y,0)$ and $(y,1)$, respectively. With this topology $X$ is a compact Hausdorff zero-dimensional space which is first countable but not second countable (and so not metrizable). See \cite[p.~44]{Ell} or \cite[p.~28]{Aus}.

Now the discrete topological abelian group $T$ acts on $X$ by the action map
\begin{gather*}
\pi\colon T\times X\rightarrow X;\quad (t,(y,j))\mapsto (y+t,j)\ \forall t\in T\textrm{ and }(y,j)\in X.
\end{gather*}
Given any $x=(y,j)\in X, j=0,1$ the orbit $Tx=Y_j$ and so $\overline{Tx}=X$ and thus $(T,X,\pi)$ is minimal. If $y\in\mathbb{T}$, then $(y,0)$ is proximal to $(y,1)$ because for every $\varepsilon>0$ there is some irrational $t\in T$ such that $t(y,0)=(y+t,0)$ and $t(y,1)=(y+t,1)$ both are in $N_\varepsilon(y,0)$. On the other hand, for any $y\in\mathbb{T}$ and any rational $t\in T$, $t_n(y,j)=(y,j)$ where $t_n=nt$ for some integer $n>0$ and so $t_n(y,j)\in N_\varepsilon(y,j)$ for any $\varepsilon>0$.

This shows that $((y,0),(y,1))\in X\times X$ is a Li-Yorke chaotic pair of $(T,X,\pi)$, for any $y\in\mathbb{T}$. Moreover, $(T,X,\pi)$ is sensitive by Proposition~\ref{prop2.8} but not Li-Yorke chaotic.\hfill$\square$
\end{exa}

We notice here that if $T$ is a discrete countable abelian monoid, then $(T,X,\pi)$ is Li-Yorke chaotic under Devaney's condition together with existing fixed points (cf.~\cite[Theorem~4.9]{WLF}) if removing the $F_n$-constraint $s_n\not\in F_n$.

Now our other Li-Yorke chaotic dynamics results may be stated as follows, which are generalizations of \cite{XY} (also see \cite[Theorem~3.8 and Corollary~3.9]{AK}) from topologically strongly-mixing cascades on compact metric spaces to topologically weakly-mixing group/semigroup actions on Polish spaces.
\begin{itemize}
\item \textit{Let $(T,X,\pi)$ be a topologically weakly-mixing topological semiflow on a non-singleton Polish space $X$ with $T$ a discrete semigroup. Then it is densely Li-Yorke chaotic} (see Proposition~\ref{prop3.12} in $\S\ref{sec3.2}$).
\item \textit{Let $(T,X,\pi)$ be a topologically weakly-mixing topological flow on a non-singleton Polish space with $T$ a topological group; then it is densely Li-Yorke chaotic}. (See Proposition~\ref{prop3.21} in $\S\ref{sec3.3}$; and cf.~\cite[Theorem~1.2]{WZ} for the special case that $T$ is a countable discrete group.)
\end{itemize}

It should be noticed here that although $(T,X,\pi)$ is topologically weakly mixing, yet $(X,\pi_t)$, for $t\in T$, need not be a topologically weakly-mixing cascade. Thus our semigroup/group acting case is never a consequence of the well-known cascade case. The above propositions will be respectively proved in $\S\ref{sec3.2}$ and $\S\ref{sec3.3}$ based on a classical topological lemma of Kuratowski.

On the other hand, by using Furstenberg Intersection Lemma (cf.~Lemma~\ref{lem3.2}) and the sensitivity of topologically weakly-mixing systems (cf.~Lemma~\ref{lem3.20}), we shall prove the followings.
\begin{itemize}
\item \textit{Any topologically weakly-mixing semiflow $(T,X,\pi)$ on a non-singleton compact metric space $X$ with $T$ an abelian semigroup not necessarily discrete, it is densely Li-Yorke chaotic} (see Proposition~\ref{prop3.22} in $\S\ref{sec3.3}$).
\item \textit{Any topologically strongly-mixing flow on a non-singleton Polish space is densely Li-Yorke chaotic} (see Corollary~\ref{cor3.25} in $\S\ref{sec3.5}$).
\end{itemize}
It should be reminded again here that our Li-Yorke chaos is very sensitive to the topology of the phase group or semigroup $T$.
\subsection{Multi-dimensional Li-Yorke chaos}\label{sec1.3}
We now consider a new chaotic behavior---multi-dimensional Li-Yorke chaos, which looks more chaotic than the usual Li-Yorke chaos. Given any reference sequence $\mathcal{F}=\{F_n\}_{n=1}^\infty$ of compact subsets of the topological semigroup $T$, we shall say that:

\begin{defn}\label{def1.5}
Let $(T,X,\pi)$ be a topological semiflow on a topological space $X$ with a phase semigroup $T$. Then
\begin{itemize}
\item $(T,X,\pi)$ is called \textbf{\textit{multi-dimensional chaotic relative to $\mathcal{F}$}}, provided that one can find an \textit{infinite} subset $\Theta$ of $X$ such that
    \begin{itemize}
    \item  for any $k\ge2$ and any distinct points $x_1,\dotsc,x_k\in\Theta$ there are two sequences $\{s_n\}, \{t_n\}$ in $T$ and a point $p\in X$ with $s_n\not\in F_n$ for $n=1,2,\dotsc$ such that
    \begin{equation*}
    s_n(x_1,\dotsc,x_k)\to(x_1,\dotsc,x_k)\in X^k\quad \textrm{and} \quad t_n(x_1,\dotsc,x_k)\to (p,\dotsc,p)\in\varDelta_{X^k}
    \end{equation*}
     as $n\to\infty$. Here $\varDelta_{X^k}=\{(x,\dotsc,x)\in X^k\,|\,x\in X\}$.
    \end{itemize}
    It is called \textbf{\textit{multi-dimensional chaotic}} if it is multi-dimensional chaotic relative to any reference sequence $\mathcal{F}$ in $T$.
\item $(T,X,\pi)$ is called \textbf{\textit{multi-dimensional Li-Yorke chaotic relative to $\mathcal{F}$}} if one can further find a uncountable such $\Theta$ in the above. It is called \textbf{\textit{multi-dimensional Li-Yorke chaotic}} if it is multi-dimensional Li-Yorke chaotic relative to any reference sequence $\mathcal{F}$ in $T$.
\end{itemize}
\end{defn}

Clearly, the `2D Li-Yorke chaos' case, which is known as `strong Li-Yorke chaos' for cascades on compact metric spaces in \cite{AGHSY}, is already conceptually stronger than the usual Li-Yorke chaos introduced in $\S\ref{sec1.2}$.

We shall prove the following statements in $\S\ref{sec2.3}$:
\begin{itemize}
\item On a Polish space with $T$ an abelian \textit{topological} group (not necessarily discrete), it holds that
\begin{itemize}
\item \textit{Devaney chaos $\Rightarrow$ Multi-dimensional chaos} (see Proposition~\ref{prop2.14} and Corollary~\ref{cor2.16}).
\end{itemize}
\item On a Polish space with no isolated points,
\begin{itemize}
\item \textit{Devaney chaos $\Rightarrow$ Multi-dimensional Li-Yorke chaos} (see Proposition~\ref{prop2.17} for the case $T=(\mathbb{R}_+,+)$ and Proposition~\ref{prop2.21} for \textit{discrete} abelian group $T$).
\end{itemize}
\item[] It should be mentioned that the topology of the abelian phase group or semigroup $T$ is important for Devaney chaos here; this is because different topologies of $T$ will define periodic orbits with different topological structures.
\end{itemize}

It is a known fact that a topologically weakly-mixing flow may have no periodic points differently with a Devaney chaotic system. For this kind of dynamics, we shall prove in $\S\ref{sec3.2}$:
\begin{itemize}
\item \textit{On a non-singleton compact metric space $X$, every minimal topologically weakly-mixing flow $(T,X,\pi)$ with $T$ an abelian group is densely multi-dimensional Li-Yorke chaotic} (see Proposition~\ref{prop3.9}).
\item  \textit{Any topologically weakly-mixing semiflow $\mathbb{R}_+\curvearrowright_\pi X$ on a non-singleton Polish space $X$ is densely multi-dimensional Li-Yorke chaotic} (see Proposition~\ref{prop3.10}).
\item \textit{Every topologically weakly-mixing semiflow on a non-singleton Polish space with discrete abelian phase semigroup is densely multi-dimensional Li-Yorke chaotic} (see Proposition~\ref{prop3.13}).
\end{itemize}

As the Li-Yorke chaos in $\S\ref{sec1.2}$, the multi-dimensional chaos also depends upon the topology of the phase semigroup/group $T$.

\begin{rem}
The above multi-dimensional Li-Yorke chaos is comparable with the known concept `uniform chaos' \cite[Definition~2.13]{AGHSY} defined for a cascade on a compact metric space with the discrete phase semigroup $\mathbb{Z}_+$.
\end{rem}
\subsection{Standing notation}\label{sec1.4}
In this paper, unless stated otherwise, $X$ denotes a uniform topological space (completely regular space) with a compatible symmetric uniformity $\mathscr{U}_X\subset 2^{X\times X}$ and $T$ denotes a multiplicative topological group or semigroup. By $e$ it means the neutral element of $T$ if $T$ is a monoid.
\begin{itemize}
\item A subset $A$ of $T$ is called \textbf{\textit{thick}} if for any compact subset $K$ of $T$, one can find some $t\in T$ with $Kt\subseteq A$. If $A$ is a thick subset of a non-trivial semigroup $T$, then $A\not=\{e\}$.
\item A subset $B$ of $T$ is called \textbf{\textit{syndetic}} provided that there is a compact subset $K$ of $T$ such that $Kt\cap B\not=\emptyset\ \forall t\in T$. (This is very different with \cite[Def.~3.1]{SKBS} requiring $T=K^{-1}B=KB$. See \cite{CD} for counterexamples.)

    Then $B$ is syndetic if and only if $B\backslash\{e\}$ is syndetic in a non-trivial $T$. Indeed, let $B$ be syndetic and $K$ a compact set in $T$ with $Kt\cap B\not=\emptyset\ \forall t\in T$. Given any $b\in B$ with $b\not=e$, we have $(K\cup bK)t\cap(B\setminus\{e\})\not=\emptyset\ \forall t\in T$. This follows the conclusion.
\item Any set $B$ is syndetic in $T$ if and only if $B\cap A\not=\emptyset$ for any thick subset $A$ of $T$.
\end{itemize}

Notice here that if the phase semigroup $T$ is not a monoid, we can always add forcedly a two-sided neutral element $e$ to $T$ with no serious false aftermath.

Finally, for any topological semiflow $(T,X,\pi)$ we should bear in mind that it is possible that distinct elements $s$ and $t$ of $T$ may induce the same transition map: $\pi_t(x)=\pi_s(x)$ for all $x\in X$.
\begin{itemize}
\item A point $x$ is called \textbf{\textit{almost periodic}} or \textit{uniformly recurrent} for a semiflow $(T,X,\pi)$ if for any neighborhood $U$ of $x$,
\begin{equation*}
N_T(x,U)=\{t\in T\,|\,tx\in U\}
\end{equation*}
is syndetic in $T$. See \cite{GH,Fur,Aus,CD}.

\item A point $x$ is called a \textbf{\textit{periodic point}} for $(T,X,\pi)$ if the stabilizer $S_x=\{t\in T\,|\,tx=x\}$ is syndetic in $T$. (This is much weaker than Schneider et al.~\cite[Def.~3.4]{SKBS} and \cite{CC,Ar}.)

\item We shall say that $(T,X,\pi)$ is \textbf{\textit{uniformly almost periodic}} if for any $\varepsilon\in\mathscr{U}_X$ there exists a syndetic subset $A$ of $T$ such that $Ax\subseteq\varepsilon[x]$ for every $x\in X$. See \cite{GH, Aus}.
\end{itemize}
The above various recurrence definitions are also valid for a semiflow on any topological phase space $X$.

We have mentioned that the topology of the phase semigroup/group $T$ is important for our chaotic dynamics. For example, let $T\curvearrowright_\pi X$ be a topological flow on a non-singleton compact metric space, which is pointwise periodic (i.e. every point of $X$ is periodic). Then if here $T$ is under the discrete topology, then $T\curvearrowright_\pi X$ is never weakly mixing. However, we do not know whether or not this is also true for general non-discrete topological group.
\section{Sensitivity of M-semiflows and Devaney chaos implying Li-Yorke chaos}\label{sec2}
This section will be mainly devoted to proving that M-semiflow is sensitive to initial conditions stated in $\S\ref{sec1.1}$ and to introducing Devaney chaos and showing ``Devaney chaos $\Rightarrow$ sensitivity'' for topological semiflows in $\S\ref{sec2.2}$. And then we will show ``Devaney chaos $\Rightarrow$ Multi-dimensional Li-Yorke chaos'' for topological flows with abelian phase groups in $\S\ref{sec2.3}$.

\subsection{Sensitivity of M-semiflows}\label{sec2.1}
Let $(T,X,\pi)$ be any given topological semiflow on a topological space $X$ with arbitrary topological phase semigroup $T$.

In the sequel, $t^{-1}$ denotes the inverse $\pi_t^{-1}$ of the transition map $\pi_t\colon x\mapsto tx$ of $X$ to itself, not the inverse element of $t$ in $T$, for $T$ is only a semigroup. We need to recall and introduce some basic notation as follows:
\begin{defn}
\begin{enumerate}
\item[(1)] $(T,X)$ is called \textbf{\textit{topologically transitive}} provided that for any nonempty open sets $U,V$ of $X$, $U\cap t^{-1}V\not=\emptyset$ for some $t\in T$. See, e.g.,~\cite{GH, KM}.

\item[(2)] $(T,X)$ is called an \textbf{\textit{M-semiflow}} if $(T,X)$ is \textit{topologically transitive} satisfying the so-called \textit{Bron\v{s}te\v{\i}n condition}; that is, the set of all almost periodic points of $(T,X)$ is dense in $X$. See, e.g.,~\cite{GW,KM,WLF}, and \cite{WCF15} for some characterizations of M-semiflow.

\item[(3)] $(T,X)$ is called \textbf{\textit{point-transitive}} in case there exists a point $x_0\in X$, written $x_0\in\textit{Trans}(T,X)$, such that the orbit $Tx_0:=\pi(T,x_0)$ is dense in $X$.
    See \cite[Def.~9.02]{GH} for the group case and see \cite[Def.~1.6(3)]{EEN} and \cite[Def.~3.1(2)]{KM} for the semigroup case.

\item[(4)] $(T,X)$ is called \textbf{\textit{syndetically transitive}} provided that for any nonempty open sets $U,V$ in $X$,
$N_T(U,V)=\left\{t\in T\,|\,U\cap t^{-1}V\not=\emptyset\right\}$
is syndetic in $T$. See \cite[Def.~9.02]{GH} and \cite[Def.~2.5]{WLF}.
\end{enumerate}
\end{defn}

Clearly, syndetically transitive implies topologically transitive.
It should be noted that in our general semigroup situation (even on a compact metric space $X$), ``topologically transitive'' and ``point-transitive'' are conceptually independent each other. Of course, $\overline{\textit{Trans}(T,X)}=X$ implies topologically transitive for any semiflow $(T,X)$.

\begin{BF}[{cf.~\cite[Proposition~3.2]{KM}}]\label{fac1}
Let $(T,X)$ be a topological semiflow with \textit{Polish phase space} $X$; then the following properties are equivalent:
\begin{enumerate}
\item[($\textrm{tr}_1)$] $(T,X)$ is topologically transitive.

\item[($\textrm{tr}_2)$] $\textit{Trans}(T,X)$ is a dense $G_\delta$-set of $X$.
\end{enumerate}
\end{BF}

\begin{BF}\label{fac2}
Let $(T,X)$ be a topological semiflow on a compact metric space $X$. Then the following properties are pairwise equivalent.
\begin{enumerate}
\item[$(1)$] $(T,X)$ is an M-semiflow.
\item[$(2)$] $\overline{\textit{Trans}(T,X)}=X$ and for any $x\in\textit{Trans}(T,X)$ and any nonempty open set $U$ in $X$, $N_T(x,U)$ is \textbf{\textit{piecewise syndetic}} (i.e., there exists a syndetic subset $S$ of $T$ such that for any compact set $A\subseteq S$, $At_A\subseteq N_T(x,U)$ for some $t_A\in N_T(x,U)$).
\end{enumerate}
Whenever $T$ is abelian, then each of $(1)$ and $(2)$ is equivalent to
\begin{enumerate}
\item[$(3)$] There exists a point $x_0\in X$ such that $N_T(x_0,U)$ is piecewise syndetic in $T$ for each nonempty open subset $U$ of $X$.
\end{enumerate}
\end{BF}

\begin{proof}
$(1)\Rightarrow(2)\Rightarrow(3)$ are trivial. Here there is no need that $T$ is abelian in $(2)\Rightarrow(3)$.

$(3)\Rightarrow(1)$: Let $T$ be abelian and let $x_0\in X$ satisfy that $N_T(x_0,U)$ is piecewise syndetic in $T$ for each nonempty open set $U\subset X$. Since $X$ is a Polish space, $\textit{Trans}(T,X)$ is a $G_\delta$ set. In addition, $Tx_0\subset\textit{Trans}(T,X)$. Then $(T,X)$ is topologically transitive by Basic Fact~\ref{fac1}.

On the other hand, take an open set $V$ with $\overline{V}\subset U$ and $V\not=\emptyset$. Since $N_T(x_0,V)$ is piecewise syndetic in $T$, we can find a syndetic set $S$ in $T$ such that for any compact set $A\subseteq S$ there is some $t_A\in N_T(x_0,V)$ with $At_A\subseteq N_T(x_0,V)$. Since $\{t_Ax_0\,|\,A\subseteq S\textrm{ compact}\}$ is a net in $X$, then there is a subnet $\{t_{A^\prime}\}$ of $\{t_Ax_0\}$ such that $t_{A^\prime}x_0\to y\in\overline{V}$ with $Sy\subseteq\overline{V}$. Let $\Delta$ be a $T$-minimal subset of $\overline{Ty}$. We can assert $\Delta\cap U\not=\emptyset$. Otherwise, there is a neighborhood $N$ of $\Delta$ with $\overline{V}\cap N=\emptyset$ such that $N_T(y,N)$ is thick and so $S\cap N_T(y,N)=\emptyset$; but this is a contradiction.

Moreover, $(2)\Rightarrow(1)$ may be completed by the second part of the proof of $(3)\Rightarrow(1)$.
\end{proof}

\begin{BF}\label{fac3}
Let $(T,X)$ be an M-semiflow on a Polish space $X$ with $T$ abelian. Then $\textit{Trans}(T,X)$ is an invariant residual set in $X$.
\end{BF}

\begin{proof}
This follows from Basic Fact~\ref{fac1} and the first part of the above proof of $(3)\Rightarrow(1)$ of Basic Fact~\ref{fac2}.
\end{proof}

\begin{BF}\label{fac4}
Let $(T,X,\pi)$ be any topological semiflow with $e\in T$. Then $(T,X,\pi)$ is topologically transitive iff every invariant open nonempty subset of $X$ is dense. So topologically transitive is also called topologically ergodic in some literature like \cite{Fur67,G76}.
\end{BF}

Let $(T,X,\pi)$ be a semiflow on a metric space $(X,d)$ with phase semigroup $T$; then we define a new metric on $X$ as follows:
\begin{gather*}
d_T(x,y)={\sup}_{t\in T}d(tx,ty)\quad \forall x,y\in X.
\end{gather*}
Our next result relates the metric $d_T$ to the sets of transitive and recurrent points. It is just \cite[Theorem~2.6]{AAB} in the special case that $T$ is the transition semigroup of a $\mathbb{Z}_+$-action.

\begin{prop}\label{pro2.2}
Let $(T,X)$ be a topological semiflow on a Polish space $(X,d)$. Then the metric $d_T$ is complete and with respect to $d_T$,
\begin{enumerate}
\item[$(1)$] $\textit{Trans}(T,X)$ is closed (but not necessarily $T$-invariant unless $X$ has no isolated point and $T$ is an $F$-semigroup (i.e. $T\backslash{Tt}$ is finite for any $t\in T$; cf.~\cite[Def.~2.1]{KM}));
\item[$(2)$] the almost periodic points of $(T,X)$ form a closed set.
\end{enumerate}
Moreover, if $(T,X)$ is equicontinuous, then $\textit{Trans}(T,X)$ is a closed subset of $X$.
\end{prop}

\begin{proof}
This follows from a slight improvement of the proof of \cite[Theorem~2.6]{AAB}. Thus we omit the details here.
\end{proof}
\subsubsection{Non-minimal M-semiflow is sensitive}
The following Lemma~\ref{lem2.3} generalizes \cite[Theorem~3.3]{WLF} from abelian acting semigroup to any acting semigroup.

\begin{lem}\label{lem2.3}
If $(T,X)$ is an M-semiflow with $X$ an arbitrary topological space, then it is syndetically transitive.
\end{lem}

\begin{proof}
Given any two nonempty open sets $U,V$ in $X$, there exists some element $s\in T$ such that $V_1:=U\cap s^{-1}V\not=\emptyset$ for $(T,X)$ is topologically transitive.
Now take an almost periodic point $x\in V_1$. Then for $t\in N_T(x,V_1)$, we have
$x\in t^{-1}V_1$; so $x\in(st)^{-1}V$ and then $U\cap(st)^{-1}V\not=\emptyset$.
Thus $sN_T(x,V_1)\subseteq N_T(U,V)$. Since $N_T(x,V_1)$ and then $sN_T(x,V_1)$ both are syndetic in $T$, hence $N_T(U,V)$ is syndetic in $T$. This proves Lemma~\ref{lem2.3}.
\end{proof}

It is interesting to note that the phase space $X$ need not be compact in Lemma~\ref{lem2.3} and the following Lemma~\ref{lem2.4}.

\begin{lem}\label{lem2.4}
Let $X$ be a uniform space with a uniformity $\mathscr{U}_X$. If $(T,X)$ is syndetically transitive and is non-sensitive to initial conditions, then it is minimal.
\end{lem}

\begin{proof}
Since $(T,X)$ is not sensitive, there is an $x_0\in X$ at which $(T,X)$ is equicontinuous. Further by the topological transitivity, it is easy to check that $x_0$ is a transitive point of $(T,X)$ (see, e.g.,~\cite[Lemma~3.3]{KM} or \cite[Exercise~2.7]{Aus}). Thus we only need to show that $x_0$ is almost periodic for $(T,X)$ by the fact that a uniform space is regular.
For this, let $\varepsilon\in\mathscr{U}_X$ be arbitrary; then by syndetically transitivity and equicontinuity at $x_0$, we can take an index $\delta\in\mathscr{U}_X$ such that
$N_T(\delta[x_0],\delta[x_0])\subseteq N_T(x_0,\varepsilon[x_0])$, whence $N_T(x_0,\varepsilon[x_0])$ is syndetic in $T$. This concludes the proof of Lemma~\ref{lem2.4}.
\end{proof}

Now the following result completes the proof of Proposition~\ref{prop2.8}.(2), which means that Devaney chaos implies sensitive at the same time. It should be mentioned that the non-minimality of $(T,X,\pi)$ is essential for our conclusion.

\begin{prop}\label{prop2.5}
Let $(T,X,\pi)$ be a non-minimal M-semiflow on a uniform space $(X,\mathscr{U}_X)$; then it is sensitive to initial conditions.
\end{prop}

\begin{proof}
By Lemma~\ref{lem2.3}, $(T,X,\pi)$ is syndetically transitive. So if it were not sensitive, then by Lemma~\ref{lem2.4} it is minimal. This is a contradiction.
\end{proof}

Proposition~\ref{prop2.5} generalizes \cite[Theorem~1.3]{GW} and \cite[Theorem~2.5]{AAB} from cascades on compact metric spaces to any semigroup actions on uniform spaces.

\subsubsection{M-semiflow is sensitive}
The following is a generalization of \cite[Corollary~3.2]{AK} from cascade $(X,f)$ on a compact Hausdorff space $X$ to any topological semiflow $(T,X,\pi)$ with abelian acting semigroup $T$, which is comparable with \cite[Theorem~5.6]{KM}.

\begin{prop}\label{prop2.6}
Let $(T,X,\pi)$ be a minimal topological semiflow on a uniform space $X$ such that
\begin{itemize}
\item either (1)\,---\,$T$ is abelian and $X$ is a compact Hausdorff space
\item or (2)\,---\,$T$ is an $F$-semigroup (i.e. $T\setminus Tt$ is finite for any $t\in T$).
\end{itemize}
Then $(T,X,\pi)$ is either equicontinuous or sensitive to initial conditions.
\end{prop}

\begin{proof}
Under condition (2) the statement follows from a slight improvement of the proof of \cite[Proposition~5.1]{KM}.
Next, let $(T,X,\pi)$ be non-sensitive with $T$ abelian and with $X$ a compact Hausdorff space; then there exists a point $x_0\in  X$ at which $(T,X,\pi)$ is equicontinuous. We are going to show $(T,X,\pi)$ is uniformly almost periodic.

For that, given any index $\varepsilon\in\mathscr{U}_X$, let $\delta\in\mathscr{U}_X$ be such that
$(x_0,y)\in\delta$ implies that $(sx_0,sy)\in\varepsilon$ for all $s\in T$.
Now for any $t\in N_T(x_0,\delta[x_0])$ we have $(x_0,tx_0)\in\delta$ and then
$$
(sx_0,stx_0)=(sx_0,t(sx_0))\in\varepsilon\quad\forall s\in T.
$$
Since $Tx_0$ is dense in $X$, hence
$(z,tz)\in\varepsilon\; \forall z\in X$ and $t\in A:=N_T(x_0,\delta[x_0])$; that is, $Az\subseteq\varepsilon[z]$ for each $z\in X$.
Because $N_T(x_0,\delta[x_0])$ is syndetic in $T$, $(T,X,\pi)$ is uniformly almost periodic.

Further by \cite[Theorem~1.2]{DX}, it follows that $(T,X,\pi)$ is equicontinuous.

This thus concludes the proof of Proposition~\ref{prop2.6}.
\end{proof}

Clearly our Proposition~\ref{prop2.6} is also a generalization of the Auslander-Yorke dichotomy theorem \cite[Corollary~2]{AY}.

\begin{cor}[Dichotomy theorem for abelian semigroups]\label{cor2.7}
Let $(T,X,\pi)$ be an M-semiflow (or more generally, a syndetically transitive semiflow) on a compact Hausdorff space with $T$ abelian. Then it is either minimal equicontinuous or sensitive to initial conditions.
\end{cor}

\begin{proof}
Since $(T,X,\pi)$ is an M-semiflow, then Lemma~\ref{lem2.3} follows that $(T,X,\pi)$ is syndetically transitive. Furthermore, if $(T,X,\pi)$ is not sensitive, then it is minimal by Lemma~\ref{lem2.4} and so it is equicontinuous by Proposition~\ref{prop2.5}.
\end{proof}

Proposition~\ref{prop2.5} together with the above Corollary~\ref{cor2.7} implies the following.

\begin{prop}\label{prop2.8}
Let $(T,X,\pi)$ be an M-semiflow on a uniform space $X$, which satisfies one of the following conditions:
\begin{enumerate}
\item[$(1)$] $(T,X,\pi)$ is non-equicontinuous with $T$ abelian and with $X$ a compact Hausdorff space;
\item[$(2)$] it is non-minimal.
\end{enumerate}
Then $(T,X,\pi)$ is sensitive to initial conditions.
\end{prop}

Recall that $T$ is called a \textit{C-semigroup} if $T\setminus Tt_0$, for each $t_0\in T$, is relatively compact in $T$ (see \cite{KM}). Since under the discrete topology, C-semigroup is just an F-semigroup. Thus if $(T,X)$ is an M-semiflow with $T$ a C-semigroup on a Polish space $X$, then it is
either equicontinuous or sensitive or minimal. This is just Main result of \cite{KM}. However, an abelian semigroup need not be a C-semigroup.
Therefore, Proposition~\ref{prop2.8} (Corollary~\ref{cor2.7} and Proposition~\ref{prop2.5}) also generalizes \cite[Main result]{KM}. If $T$ is assumed abelian acting on a compact metric space $X$, then Proposition~\ref{prop2.5} is a special case of \cite[Theorem~3.10]{WCF15}.

It should be noticed that if $T$ is non-abelian, then the conclusion of Proposition~\ref{prop2.6} is not necessarily true. However, we do not know whether or not the conclusion of Corollary~\ref{cor2.7} is still true if $T$ is non-abelian.

\subsection{Devaney chaos of semiflows}\label{sec2.2}
Let $\pi\colon T\times X\rightarrow X$ or $(T,X,\pi)$ be a topological semiflow on a topological space $X$ with phase semigroup $T$, where $T$ is a topological semigroup not necessarily discrete.

Following \cite[Def.~3.12]{SKBS} we now introduce Devaney chaos as follows:

\begin{defn}\label{def2.9}
$(T,X,\pi)$ is called \textbf{\textit{Devaney chaos}} if it satisfies the following conditions:
\begin{enumerate}
\item[(1)] $(T,X,\pi)$ is topologically transitive.
\item[(2)] The set of periodic points of $(T,X,\pi)$ is dense in $X$.
\item[(3)] $(T,X,\pi)$ is not minimal.
\end{enumerate}
\end{defn}

\begin{note}
Since here our ``periodic point'' is weaker than that of Schneider et al.~\cite{SKBS}, hence this chaos condition is weaker than that of~\cite[Def.~3.12]{SKBS}.
\end{note}

\begin{note}
Since $T$ need not be discrete, hence a periodic orbit here is not necessarily finite. Therefore, (1) + (2) $\not\Rightarrow$ (3) even if $X$ has no isolated points. Moreover, the most important classical $C^0$-flow on a manifold with $T=\mathbb{R}$ has no non-trivial ``discrete'' periodic orbits. Hence our Devaney chaos conditions are weaker than that of \cite{WLF, CC, Ar}.
\end{note}

\begin{note}
Clearly (1) + (2) $\Rightarrow$ $(T,X,\pi)$ is an M-semiflow.
\end{note}

\begin{note}
Condition (3) is not ignorable for our Devaney chaos, for example, the $\mathbb{R}$-acting $C^0$-flow on the unit circle $\mathbb{T}=\mathbb{R}/\mathbb{Z}$ by rotation $(t,x)\mapsto t+x\pmod{1}$ is topologically transitive with dense periodic points (in fact it is pointwise periodic) but with no interesting ``chaos'' dynamics.

Of course, if $X$ is infinite like with no isolated point and $T$ is a group or induced by a cascade with the discrete topology, then (1) + (2) $\Rightarrow$ (3).
\end{note}

\begin{note}
Since $X$ is not necessarily a Polish space, there does not need to have a transitive point in general.
\end{note}

On a Hausdorff uniform space $X$, every periodic orbit in the sense of Schneider et al.~\cite{SKBS} is a compact minimal set. And so two distinct periodic orbits $Tx$ and $Ty$ in \cite{SKBS} are far away (i.e. one can find an index $\beta\in\mathscr{U}_X$ with $\beta[Tx]\cap\beta[Ty]=\emptyset$). This is an important point for the proof of Schneider et al.~\cite{SKBS}.

However in our situation that $X$ is only a uniform space (or equivalently a completely regular space) and that a periodic point is only weakly defined in $\S\ref{sec1.4}$, although $\overline{Tx}$ is minimal for any periodic point $x$ in our sense, yet $Tx$ itself need not be a compact closed set in $X$ and so any two distinct periodic orbits are not necessarily far away.

So the arguments of Banks et al.~\cite{BB} does not work in our setting here. However, as a result of Proposition~\ref{prop2.8} (or Proposition~\ref{prop2.5}), we can obtain the following, which covers \cite{BB,CC,SKBS} and works for classical $C^0$-flows on manifolds with $T=\mathbb{R}$.

\begin{prop}\label{prop2.10}
Let $(T,X,\pi)$ be a topological semiflow on a uniform topological space $X$. If it is of Devaney chaos, then it is sensitive to initial conditions and thus $X$ has no isolated points.
\end{prop}

It should be noted that although Proposition~\ref{prop2.10} is stronger than \cite[Theorem~4.7]{SKBS}, yet our approaches (needing only Lemmas~\ref{lem2.3} and \ref{lem2.4}) are completely different with and much more simpler than those introduced by Banks et al. in \cite{BB, CC, SKBS}.

\subsection{Devaney chaos implies Li-Yorke chaos}\label{sec2.3}
The following main criterion---Proposition~\ref{prop2.14}---implies that ``Devaney chaos $\Rightarrow$ Multi-dimensional chaos'' on Polish spaces. Here $\textit{Int}_XA$ will denote the interior of a subset $A$ relative to the topological space $X$.

In some literature like \cite{WLF}, a point $x$ is referred to as a recurrent point for a semiflow $(T,X,\pi)$ if $\exists$ an infinite net $\{t_n\}$ in $T$ with $t_nx\to x$. However, if the neutral element $e$ is a limit point of $T$ in Ellis' semigroup $E(X)$, then all points are recurrent and this definition becomes trivial even for discrete countable $T$.

For our convenience, we first need to introduce a notation to overcome this drawback, which is much more stronger than the same named notion available in the literature.

\begin{defn}\label{def2.11}
A point $x$ of $X$ is called a \textbf{\textit{recurrent point}} of a semiflow $(T,X,\pi)$ on a topological space $X$ with topological phase semigroup $T$, if for any neighborhood $U$ of $x$ and any compact proper subset $K$ of $T$, one can find some $t\in T\setminus K$ with $tx\in U$.
\end{defn}

Clearly this recurrence is very valid for non-compact phase semigroup $T$. It is easy to check that if $x$ is a recurrent point of a flow $(T,X,\pi)$, then each $y\in Tx$ is also a recurrent point of $(T,X,\pi)$.
It should be noted that a transitive point is not necessarily a recurrent point for a flow on a Polish space; see, e.g., Example~\ref{e1}.

The following simple observation will be useful in the proof of our multi-dimensional chaos criterion Proposition~\ref{prop2.14}.

\begin{lem}\label{lemII.12}
Let $(T,X,\pi)$ be a topological semiflow on a Hausdorff topological space with non-compact phase semigroup $T$. If $x_0\in\textit{Trans}(T,X)$ is such that $\textit{Int}_X(Tx_0)=\emptyset$, then $x_0$ is a recurrent point of $(T,X,\pi)$.
\end{lem}

\begin{proof}
Let $U$ be any open neighborhood of $x_0$ and let $K$ be any compact subset of $T$. Since $\textit{Int}_X(Tx_0)=\emptyset$, hence $U\setminus Kx_0$ is a nonempty open subset of $X$. Thus there is some $t\in T\setminus K$ with $tx_0\in U$. This proves Lemma~\ref{lemII.12}.
\end{proof}

Notice that the ``empty-interior'' condition is non-ignorable here. Let's see a simple counterexample as follows.
Let $T=\mathbb{R}$ and $X=\mathbb{R}$; then the $C^0$-flow $\pi\colon T\times X\rightarrow X$ given by $(t,x)\mapsto t+x$ is such that $\textit{Trans}(T,X)=X$ and $Tx=X$ for all $x\in X$. However, this flow has no recurrent points. In addition, see Example~\ref{e1}. On the other hand, it is easy to verify that:
\begin{itemize}
\item If $f\colon X\rightarrow X$ is topologically transitive on a Hausdorff space $X$ with no isolated points, then $\{f^n(x_0)\,|\,n=0,1,2,\dotsc\}$ has empty interior for any $x_0\in X$.
\end{itemize}

Let $W$ be any space and $n\ge2$. Recall that a nonempty compact set $F\subset W$ is said to be \textit{independent} in $R\subset W^n$, written $F\in J(R)$, if for every distinct points $w_1,\dotsc,w_n\in F$ the $n$-dimensional point $(w_1,\dotsc,w_n)$ never belongs to $R$ (see \cite{Ku}).

Finally we will need a classical topological lemma.

\begin{lem}[{Kuratowski-Mycielski theorem; cf.~\cite{Ku}}]\label{lem2.13}
Let $W$ be a complete metric space and let $R_1,R_2,R_3,\dotsc$ be a sequence of $F_\sigma$-sets $R_k\subset W^{n_k}$ of the first category. Then $\bigcap_kJ(R_k)$ is a dense $G_\delta$-set in the space $\mathscr{K}(W)$ of all nonempty compact subsets of $W$ (with the Hausdorff-distance topology).

If $W$ has no isolated points, then each $J(R_k)$ above may consist of Cantor sets (homeomorphisms of the Cantor discontinuum) of $W$ independent in $R_k$.
\end{lem}

The following criterion of chaos is comparable with the known criterion for uniform chaos of $\mathbb{Z}_+$-actions \cite[Theorem~3.1 and Theorem~4.10]{AGHSY} using different approaches.

However, it should be noticed that since now the induced semiflow $(S,Y)$ does not need to have a fixed point in the following proof, hence we cannot expect to find a multi-dimensional Li-Yorke chaotic set for $(S,Y)$ instead of $(T,X,\pi)$.

\begin{prop}\label{prop2.14}
Let $(T,X,\pi)$ be a topological flow on a Polish space $X$ with $T$ an abelian topological group. Assume
\begin{itemize}
\item[$(\mathrm{a})$] $(T,X,\pi)$ is topologically transitive with $\textit{Int}_X(Tx)=\emptyset$ for some $x\in\textit{Trans}(T,X)$;
\item[$(\mathrm{b})$] there exists a periodic point $p$ in $X$.
\end{itemize}
Then, for any sequence of compact subsets of $T\colon F_1\subset F_2\subset F_3\subset\dotsm$, one can find an infinite subset $\Theta$ of $X$ such that for any $k\ge2$ and any distinct points $x_1,\dotsc, x_k\in\Theta$, there are two sequences $t_n\in T$ and $s_n\in T\setminus F_n$ with
\begin{gather*}
t_n(x_1,\dotsc,x_k)\to(p,\dotsc,p)\in\varDelta_{X^k}\quad \textrm{and}\quad s_n(x_1,\dotsc,x_k)\to(x_1,\dotsc,x_k)\in X^k
\end{gather*}
as $n\to\infty$. That is, $(T,X,\pi)$ is of multi-dimensional chaos.
\end{prop}

\begin{proof}
Let $k\ge2$ be any given integer.
We first note that $\textit{Trans}(T,X)$ is a dense $G_\delta$-set in $X$, for $X$ is a Polish space. Let
\begin{gather*}
S=S_p=\{t\in T\,|\,tp=p\}
\end{gather*}
be the stabilizer of the periodic point $p$. Clearly, $S$ is a closed syndetic subgroup of $T$ by definition and so it is not compact; for otherwise, $T$ is compact and so $Tx=X$ for each $x\in\textit{Trans}(T,X)$ contradicting condition (a).

Let $x_0\in\textit{Trans}(T,X)$ be any given with $\textit{Int}_X(Tx_0)=\emptyset$ by condition (a); and now for simplicity, we write $Y=\textrm{cls}_X(Sx_0)$. We note that it is easy to check that
\begin{itemize}
\item $Y$ is an infinite Polish space with the topology inherited from $X$.\footnote{There does not need to exist a fixed point in $Y$ for $(S,Y)$, since $T$ need not be discrete. This is an interesting point different with \cite{Mai04, AGHSY}.} \hfill (\ref{prop2.14}.1)

    (If $Sx_0$ is finite, then $Tx_0$ is compact and $Tx_0=\overline{Tx_0}=X$ contradicting condition (a).)
\end{itemize}
Let $Y^k=Y\times \dotsm\times Y$ ($k$-fold product). We define a $k$-ary proximal relation:
\begin{gather*}
P_k(T,Y)=\left\{(x_1,\dotsc,x_k)\in Y^k\,|\,\exists t_n\in T\textrm{ s.t. }t_n(x_1,\dotsc,x_k)\to(p,\dotsc,p)\textrm{ as }n\to\infty\right\}.
\end{gather*}
(Note here that $t_n\in T$ not necessarily requiring $t_n\in S_p$.) Since
\begin{gather*}
P_k(T,Y)=\bigcap_{n=1}^\infty\left(Y^k\cap T^{-1}B_{1/n}(p,\dotsc,p)\right),
\end{gather*}
hence $P_k(T,Y)$ is a $G_\delta$-set in $Y^k=Y\times \dotsm\times Y$. Next, let $V_1, \dotsc, V_k$ be any nonempty open subsets of $Y$. Then there are elements $s_1,\dotsc,s_k\in S$ with $s_1x_0\in V_1, \dotsc, s_kx_0\in V_k$; this is because $Y=\overline{Sx_0}$. On the other hand, by $\overline{Tx_0}=X$, it follows that there is a sequence $\{t_n\}$ in $T$ with $t_nx_0\to p$. Therefore,
\begin{gather*}
t_n(s_1x_0,\dotsc,s_kx_0)=(s_1t_nx_0,\dotsc,s_kt_nx_0)\to(s_1p,\dotsc,s_kp)=(p,\dotsc,p)
\end{gather*}
as $n\to\infty$. This shows that $P_k(T,Y)$ is dense in $Y^k$. Thus, we have concluded that
\begin{itemize}
\item $P_k(T,Y)$ is a dense $G_\delta$-set in $Y^k$.\hfill (\ref{prop2.14}.2)
\end{itemize}
In other words, the $k$D-proximal relation $P_k(T,Y)$ is a dense $G_\delta$-set restricted to $Y^k$.

Since $F_n$ is compact and $T$ is not compact (otherwise $S$ is compact), hence $T\setminus F_n$ is not empty for all $n=1,2,\dotsc$. Let $d$ be the metric on $X$ and set
\begin{gather*}
U_n=\left\{(x_1,\dotsc,x_k)\in Y^k\,|\,\exists t\in T\setminus F_n\textrm{ s.t. }d(tx_1,x_1)+\dotsm+d(tx_k,x_k)<1/n\right\}.
\end{gather*}
(Note that we only require $t\in T\setminus F_n$ not $t\in S\setminus F_n$.) Then $U_n$ is an open of $Y^k$ and so
\begin{gather*}
R_k(T,Y):=\bigcap_{n=1}^\infty U_n
\end{gather*}
is a $G_\delta$-set in $Y^k$. Moreover, since $\textit{Int}_X(Tx_0)=\emptyset$ and $x_0$ is a transitive point of $(T,X,\pi)$ from $(\mathrm{a})$, then by Lemma~\ref{lemII.12} it follows that \begin{itemize}
\item There exists a sequence $s_n\in T\setminus F_n, n=1,2,\dotsc$ such that $s_nx_0\to x_0$ as $n\to\infty$. \hfill (\ref{prop2.14}.3)
\end{itemize}
Next we will prove the following assertion:
\begin{itemize}
\item $R_k(T,Y)$ is a dense $G_\delta$-set in $Y^k$.\hfill (\ref{prop2.14}.4)
\end{itemize}
Indeed, for any nonempty open sets $U_1,\dotsc, U_k$ in $Y$, we can choose elements $\tau_1,\dotsc,\tau_k\in S$ with $\tau_1x_0\in U_1, \dotsc, \tau_kx_0\in U_k$. Thus by (\ref{prop2.14}.3), it follows that
\begin{gather*}
s_n(\tau_1x_0,\dotsc,\tau_kx_0)=(\tau_1s_nx_0,\dotsc,\tau_ks_nx_0)\to(\tau_1x_0,\dotsc,\tau_kx_0)
\end{gather*}
as $n\to\infty$. Thus (\ref{prop2.14}.4) holds.

We will need to use Lemma~\ref{lem2.13} with $W=Y$ and $R_k=Y^k\setminus\left(P_k(T,Y)\cap R_k(T,Y)\right)$ for all $k=1,2,\dotsc$. Clearly, $Y$ is an infinite complete metric space by (\ref{prop2.14}.1). By (\ref{prop2.14}.2) and (\ref{prop2.14}.4), it follows that each $R_k\subset Y^k$ is an $F_\sigma$-set of the first category. Then by Lemma~\ref{lem2.13}, it follows that there exists an infinite subset, say $\Theta$, of $Y$, such that for any distinct points $x_1,\dotsc,x_k\in\Theta$, the $k$D-point $(x_1,\dotsc,x_k)\in P_k(T,Y)\cap R_k(T,Y)$.

The proof of Proposition~\ref{prop2.14} is thus completed.
\end{proof}

\begin{cor}
Let $(T,X,\pi)$ be a topological flow on a Polish space $X$ with $T$ an abelian topological group. Assume
$(T,X,\pi)$ is a non-minimal M-flow and there exists a periodic point $p$ in $X$.
Then $(T,X,\pi)$ is of multi-dimensional chaos.
\end{cor}

\begin{proof}
This follows from Proposition~\ref{prop2.14} and that every non-minimal M-flow satisfies condition (a) of Proposition~\ref{prop2.14}.
\end{proof}

\begin{cor}\label{cor2.16}
Let $(T,X,\pi)$ be a topological flow on a Polish space $X$ with $T$ an abelain topological group. If it is Devaney chaos (in the sense of Def.~\ref{def2.9}), then it is of multi-dimensional chaos.
\end{cor}

\begin{proof}
If there were some point $x_0\in\textit{Trans}(T,X)$ with $\textit{Int}_X(Tx_0)\not=\emptyset$, then by Def.~\ref{def2.9}-(2) it follows that $x_0$ is a periodic point of $(T,X,\pi)$ and further $(T,X,\pi)$ is minimal. This is a contradiction to Def.~\ref{def2.9}-(3). Thus Corollary~\ref{cor2.16} follows from Proposition~\ref{prop2.14}.
\end{proof}

Therefore, Devaney chaos implies the multi-dimensional chaos. Another important special case of Proposition~\ref{prop2.14} is the following.

\begin{prop}\label{prop2.17}
Let $\pi\colon \mathbb{R}_+\times X\rightarrow X$ be a classical $C^0$-semiflow on a Polish space $X$ such that
\begin{itemize}
\item[$(\mathrm{a})$] $\mathbb{R}_+\curvearrowright_\pi X$ is topologically transitive with $\textit{Int}_X(Tx_0)=\emptyset$ for some $x_0\in\textit{Trans}(X,\mathbb{R}_+)$,
\item[$(\mathrm{b})$] there exists a periodic point $p$ in $X$.
\end{itemize}
Then $\mathbb{R}_+\curvearrowright_\pi X$ is multi-dimensional Li-Yorke chaotic. Particularly, if $\mathbb{R}_+\curvearrowright_\pi X$ is Devaney chaotic, then it is multi-dimensional Li-Yorke chaotic.
\end{prop}

\begin{proof}
Let $x_0\in\textit{Trans}(X,\mathbb{R}_+)$ be such that $\textit{Int}_X(\mathbb{R}_+x_0)=\emptyset$.
Thus $x_0$ is a recurrent non-periodic point of $(\mathbb{R}_+,X)$ by Lemma~\ref{lemII.12}. Next without loss of generality, let the stabilizer of $p$ is $S_p=\tau\mathbb{Z}_+$ for some $\tau>0$ and then \cite[Theorem~1.4]{Fur} it follows that $x_0$ is also a recurrent non-periodic point of $(\tau\mathbb{Z}_+,X)$. Let $Y=\overline{\pi(\tau\mathbb{Z}_+,x_0)}$. Since each point $y_0\in\pi(\tau\mathbb{Z}_+,x_0)$ is also a recurrent non-periodic point for $(\tau\mathbb{Z}_+,Y)$, thus $Y$ has no isolated points. Then the same argument of Proposition~\ref{prop2.14} concludes the proof of Proposition~\ref{prop2.17}.
\end{proof}

We note that in the proof of Proposition~\ref{prop2.17}, although $\tau\mathbb{Z}_+\curvearrowright_\pi Y$ is topologically transitive, yet we cannot guarantee `$Y$ contains a fixed point' here; this is because as $t_nx_0\to p$ where we write $t_n=\tau i_n+(\tau-r_n)\to\infty$ with $i_n\in\mathbb{Z}_+, 0<r_n\le \tau$, there is no the equality
\begin{gather*}
\lim_{n\to\infty}\pi(r_n+t_n,x_0)=\lim_{k\to\infty}\pi_{r_k}\left(\lim_{n\to\infty}\pi_{t_n}x_0\right)
\end{gather*}
for Ellis' semigroup $E(T,X)$ is never a topological semigroup in general. In view of this reason, the criterion for chaos introduced for $\mathbb{Z}_+$-acting cascade \cite[Theorem~3.1]{AGHSY} does not work here.

However, there is no such an obstruction for $\mathbb{Z}_+$-actions. The discrete-time version of Proposition~\ref{prop2.17} is also of independent interests and importance.
\begin{itemize}
\item \textit{Let $\pi\colon \mathbb{Z}_+\times X\rightarrow X$ be a semiflow on a Polish space $X$ such that
\begin{itemize}
\item[$(\mathrm{a})$] $\mathbb{Z}_+\curvearrowright_\pi X$ is topologically transitive with $\textit{Int}_X(\mathbb{Z}_+x_0)=\emptyset$ for some $x_0\in\textit{Trans}(X,\mathbb{Z}_+)$,
\item[$(\mathrm{b})$] there exists a periodic point $p$ in $X$.
\end{itemize}
Then $\mathbb{Z}_+\curvearrowright_\pi X$ is multi-dimensional Li-Yorke chaotic.}

\item \textit{Let $\pi\colon \mathbb{Z}_+\times X\rightarrow X$ be a semiflow on a Polish space $X$ with no isolated points such that
\begin{itemize}
\item[$(\mathrm{a})$] $\mathbb{Z}_+\curvearrowright_\pi X$ is topologically transitive,
\item[$(\mathrm{b})$] there exists a periodic point $p$ in $X$.
\end{itemize}
Then $\mathbb{Z}_+\curvearrowright_\pi X$ is multi-dimensional Li-Yorke chaotic.}

(It should be noted that this is stronger than Huang-Ye~\cite[Theorem~4.1]{HY} which only captures the classical Li-Yorke chaos but not our multi-dimensional Li-Yorke chaos. However, this is a consequence of \cite[Corollary~3.2-(6)]{AGHSY} by different ways.)
\end{itemize}

In our conclusion of Proposition~\ref{prop2.14}, the chaotic set $\Theta$ is not necessarily to be uncountable; this is caused by that we cannot prove $Y$ in (\ref{prop2.14}.1) has no isolated points in the case where $T$ is not discrete.

In addition, the condition that $\textit{Int}_X(Tx_0)=\emptyset$ for some transitive point is important for our conclusion as is shown by Example~\ref{e1} and the following.

\begin{exa}
Let $X=\mathbb{Z}\cup\{\infty\}$ be the one-point compactification of $\mathbb{Z}$ and let $T=\mathbb{Z}$. Then $\pi\colon(t,x)\mapsto x+t$ induces a non-minimal topologically transitive flow with a unique fixed point $\infty$. Clearly, $(T,X,\pi)$ has no Li-Yorke chaotic pair.
\end{exa}

Now, motivated by Huang and Ye~\cite{HY} for cascade on compact metric spaces, we can put forward the following

\begin{Q}\label{q2.19}
Does Devaney chaos imply the multi-dimensional Li-Yorke chaos for any topological semiflow on a Polish space?
\end{Q}

Under additional conditions (for example, $T$ is a `countable' abelian semigroup and $(T,X,\pi)$ has at least a fixed point), Wang et al. \cite[Theorem~4.9]{WLF} gave a positive solution to Question~\ref{q2.19} for the usual Li-Yorke chaos with no the constraint $s_n\not\in F_n$ in Definition~\ref{def1.2}.

As a consequence of Proposition~\ref{prop2.14}, we can easily obtain the following. Its discrete-time version is \cite[Corollary~3.2-(1)]{AGHSY}.

\begin{cor}\label{cor2.20}
Let $(T,X,\pi)$ be a topological flow on a Polish space $X$ with $T$ an abelian topological group. Suppose that
\begin{itemize}
\item[$(\mathrm{a})$] $(T,X,\pi)$ is topologically transitive such that $\textit{Int}_X(Tx)=\emptyset$ for some $x\in\textit{Trans}(T,X)$;
\item[$(\mathrm{b})$] there exists a fixed point $p$ in $X$.
\end{itemize}
Then $(T,X,\pi)$ is densely multi-dimensional Li-Yorke chaotic.
\end{cor}

\begin{proof}
Clearly $X$ has no isolated points. Under this situation, the stabilizer $S_p=T$ and so $Y=X$ in (\ref{prop2.14}.1). Hence $Y=X$ has no isolated points. Then Corollary~\ref{cor2.20} follows at once from Proposition~\ref{prop2.14} and the second part of Lemma~\ref{lem2.13}.
\end{proof}

The special case $T=\mathbb{Z}_+$ of the following Proposition~\ref{prop2.21} is just \cite[Theorem~3.3]{AGHSY} or see \cite{Mai04} using different approaches.

\begin{prop}\label{prop2.21}
Let $(T,X,\pi)$ be a topological flow on a Polish space $X$ with $T$ an abelian discrete group. If $(T,X,\pi)$ is of Devaney chaos, then $(T,X,\pi)$ is multi-dimensional Li-Yorke chaotic.
\end{prop}

\begin{proof}
Let $p$ be a periodic point with stabilizer $S=\{t\in T\,|\,tp=p\}$ which is a syndetic subgroup of $T$. Then there is a finite subset $K=\{k_1,\dotsc, k_\ell\}$ of $T$ such that $T=KS$. Let $x_0$ be any transitive point for $(T,X,\pi)$; then
\begin{gather*}
X=\overline{Tx_0}=\bigcup_{i=1}^\ell \overline{k_iSx_0}=\bigcup_{i=1}^\ell k_i\overline{Sx_0}.
\end{gather*}
Thus, $\textit{Int}_X\left(k_i\overline{Sx_0}\right)\not=\emptyset$ for some $i$. Since $\pi_{k_i}\colon X\rightarrow X$ is a homeomorphism, hence $Y=\overline{Sx_0}$ has nonempty interior relative to $X$. Thus, $Y$ contains some open ball $B_r(x_0)$ in $X$. This implies that $Y$ has no isolated points. Then the rest is same as Proposition~\ref{prop2.14}.
\end{proof}

It should be mentioned that the 2D-case of Proposition~\ref{prop2.21} implies the Li-Yorke chaos and the latter was just proved by Arai~\cite[Theorem~1.2]{Ar} by different approaches.
\section{Chaos of topologically weakly-mixing dynamics}\label{sec3}
In this section we will study the chaotic dynamics of weakly-mixing topological dynamical systems on Polish spaces with no isolated points.

In Devaney chaos, the periodic points play an important role for capturing the Li-Yorke chaos. However, a weakly-mixing flow does not need to contain any periodic point. So our approaches in this section will be very different with those introduced in the last section for Devaney chaotic flows or semiflows on Polish spaces.

We will introduce basic notions and properties as preliminaries in $\S\ref{sec3.1}$. We shall mainly study the multi-dimensional Li-Yorke chaos of weakly-mixing dynamics in $\S\ref{sec3.2}$. In $\S\ref{sec3.3}$, we will consider the Li-Yorke chaos of weakly-mixing dynamics with abelian phase semigroup or with phase group. Strongly-mixing flows will be considered in $\S\ref{sec3.5}$.

\subsection{Basic properties of weakly-mixing}\label{sec3.1}
Motivated by topologically transitive and the classical topologically weakly-mixing~(\cite[p.~26]{Fur67} and \cite[Def.~9.3, p.~183]{Fur}), we now introduce the following concepts, which are basic conditions we will assume as well.

\begin{defn}\label{def3.1}
Let $(T,X,\pi)$ be any topological semiflow on a topological space $X$ with phase semigroup $T$ and let $N_T(U,V)=\{t\in T\,|\,U\cap t^{-1}V\not=\emptyset\}$ for any two nonempty open sets $U,V$ in $X$. Then
$(T,X,\pi)$ is said to be
\begin{itemize}
\item \textbf{\textit{thickly transitive}} in case $N_T(U,V)$ is \textit{thick} in $T$ for every nonempty open subsets $U,V$ of $X$;

\item \textbf{\textit{IP-transitive}} provided that for every nonempty open subsets $U,V$ of $X$, $N_T(U,V)$ contains an \textit{\textit{IP}-set} in $T$, i.e., $\exists\langle p_i\rangle_{i=1}^\infty$ in $T$ with $p_{i_1}\dotsm p_{i_k}\in N_T(U,V)$ for any $1\le i_1<\dotsm<i_k<\infty$;

\item \textbf{\textit{topologically weakly-mixing}} if the $2$-fold product semiflow $(T,X\times X)$, where the transition map is defined by $t\colon (x,y)\mapsto(tx,ty)$ for each $t\in T$, is topologically transitive; that is to say, $N_T(U,V)\cap N_T(U^\prime, V^\prime)\not=\emptyset$ for any nonempty open sets $U,U^\prime, V,V^\prime$ in $X$.

    (We note that a thickly transitive cascade is also called topologically weakly-mixing in \cite[Def.~9.3, p.~183]{Fur} for any continuous surjection and they are equivalent in that case \cite{Fur67}.)
\end{itemize}
\end{defn}

By definitions, thickly transitive is \textit{IP}-transitive (cf.~e.g., \cite[Lemma~9.1]{Fur} for $T=\mathbb{Z}_+$) and \textit{IP}-transitive is topologically transitive.
Moreover, the following important Lemma~\ref{lem3.2} implies that topologically weakly-mixing is thickly transitive in the case that $T$ is an abelian group (Corollary~\ref{cor3.3}). Moreover, Example~\ref{e1.3} is topologically weakly-mixing (cf.~e.g.,~\cite[Corollary~II.2.2]{G76}).

\begin{lem}[Furstenberg Intersection Lemma in Abelian Case]\label{lem3.2}
A topological semiflow $(T,X,\pi)$, on a topological space $X$ with $T$ abelian, is topologically weakly-mixing if and only if
\begin{gather*}
\mathcal{F}(T,X,\pi):=\left\{A\subseteq T\,|\,\exists\textrm{ nonempty open sets }U,V\textrm{ in }X\textrm{ with }N_T(U,V)\subseteq A\right\}
\end{gather*}
is a filter on $T$; in other words, $\mathcal{F}(T,X,\pi)$ satisfies the following three conditions:
\begin{enumerate}
\item[$(1)$] $A,B\in\mathcal{F}(T,X,\pi)\Rightarrow A\cap B\in\mathcal{F}(T,X,\pi)$;
\item[$(2)$] $A\in\mathcal{F}(T,X,\pi)$ and $A\subset B\Rightarrow B\in\mathcal{F}(T,X,\pi)$; and
\item[$(3)$] $\emptyset\not\in\mathcal{F}(T,X,\pi)$.
\end{enumerate}
\end{lem}

\begin{proof}
(Following Furstenberg's framework of the proof of \cite[Proposition~II.3]{Fur67}) The sufficiency of the lemma is obvious.

Now for the necessity of the lemma, let $(T,X,\pi)$ be topologically weakly-mixing as in Def.~\ref{def3.1}.
Clearly $\mathcal{F}(T,X,\pi)$ satisfies properties (2) and (3). Now to prove property (1), let $A,B\in\mathcal{F}(T,X,\pi)$ be any given. Then $N_T(U,V)\subseteq A$ and $N_T(U^\prime,V^\prime)\subseteq B$ for some nonempty open sets $U,U^\prime,V,V^\prime$ in $X$. Take any
$t\in N_T(U,U^\prime)\cap N_T(V,V^\prime)\not=\emptyset$ and set $U^{\prime\prime}=U\cap t^{-1}U^\prime, V^{\prime\prime}=V\cap t^{-1}V^\prime$. Then for each $\tau\in N_T(U^{\prime\prime},V^{\prime\prime})$, by $t\tau=\tau t$, it follows that
\begin{gather*}
\emptyset\not=U^{\prime\prime}\cap\tau^{-1}V^{\prime\prime}=\left(U\cap t^{-1}U^\prime\right)\cap\tau^{-1}\left(V\cap t^{-1}V^\prime\right)=\left(U\cap\tau^{-1}V\right)\cap t^{-1}\left(U^\prime\cap\tau^{-1}V^\prime\right).
\end{gather*}
This implies that $U\cap\tau^{-1}V\not=\emptyset$ and $U^\prime\cap\tau^{-1}V^\prime\not=\emptyset$. Thus $N_T(U^{\prime\prime},V^{\prime\prime})\subseteq A\cap B$ and further $\mathcal{F}(T,X,\pi)$ is a filter on $T$.

This thus proves Lemma~\ref{lem3.2}.
\end{proof}

Thus, if $(T,X,\pi)$ is topologically weakly-mixing with $T$ abelian; then any $l$-fold diagonal-wise product semiflow $(T,X\times\dotsm\times X)$ is topologically transitive (cf.~\cite[Proposition~II.3]{Fur67} for cascades), and moreover, $(T,X\times\dotsm\times X)$ is also topologically weakly-mixing (cf.~\cite[Exercise~9.7]{Aus} and \cite[Theorem~1.11]{G03} for abelian groups).

The commutativity of $T$ has played a role in the proof of the necessity of Lemma~\ref{lem3.2}; otherwise, see \cite[p.~277]{W00} for counterexample. However, the necessity of Lemma~\ref{lem3.2} is just what we will need for proving our main result Proposition~\ref{prop3.22} later.

\begin{cor}\label{cor3.3}
Let $(T,X,\pi)$ be a topological semiflow with $T$ abelian on a topological space $X$. Then the following statements hold.
\begin{enumerate}
\item[$(1)$] If $(T,X,\pi)$ is topologically weakly-mixing such that $N_T(X,V)$ is thick in $T$ for any nonempty open set $V\subset X$, then it is thickly transitive.
\item[$(2)$] If $(T,X,\pi)$ is thickly transitive with $e\in T$, then $N_T(U,U)\cap N_T(V,U)\not=\emptyset$ and $N_T(U,U)\cap N_T(U,V)\not=\emptyset$ for any nonempty open sets $U,V$ in $X$.
\item[$(3)$] If $(T,X,\pi)$ is thickly transitive and its almost periodic points are dense in $X$, then $(T,X,\pi)$ is topologically weakly-mixing. (This item does not need the condition that $T$ abelian.)
\item[$(4)$] If each transition map $\pi_t\colon X\rightarrow X$ is a surjection and $e\in T$, then $(T,X,\pi)$ is topologically weakly-mixing iff it is thickly transitive.
\end{enumerate}
\end{cor}

\begin{proof}
(1): Given any nonempty open sets $U,V$ in $X$ and any finite subset $K=\{t_1,\dotsc,t_l\}$ of $T$, since $Kt_K\subseteq N_T(X,V)$ for some $t_K\in T$ and so by Lemma~\ref{lem3.2}
\begin{gather*}
N_T(U,(t_1t_K)^{-1}V)\cap N_T(U,(t_2t_K)^{-1}V)\cap\dotsm\cap N_T(U,(t_lt_K)^{-1}V)\not=\emptyset,
\end{gather*}
then we may take some $t\in T$ such that $Kt\subseteq N_T(U,V)$. Thus $N_T(U,V)$ is thick in $T$ and $(T,X)$ is thickly transitive.

(2): Let $a\in N_T(U,V)$ and $W:=U\cap a^{-1}V\,(\not=\emptyset)$. Since $N_T(W,W)$ is thick in $T$ and $\{e,a\}\subset T$, one can find some $b\in T$ with
$\{b,ab\}\subseteq N_T(W,W)$. Then
$W\cap b^{-1}W\not=\emptyset$ and $W\cap(ab)^{-1}W\not=\emptyset$.
These imply that $U\cap b^{-1}U\not=\emptyset$ and $abW\cap W\not=\emptyset$.
The latter implies that
$bV\cap U\not=\emptyset$, equivalently $V\cap b^{-1}U\not=\emptyset$.
Thus $N_T(U,U)\cap N_T(V,U)\not=\emptyset$.

Similarly, let $a\in N_T(V,U)$ and $W=V\cap a^{-1}U$. We can find $b\in T$ with $\{b,ab\}\subseteq N_T(W,W)$. Then by $W\cap b^{-1}W\not=\emptyset$, it follows that $a^{-1}U\cap b^{-1}a^{-1}U\not=\emptyset$ and so $U\cap b^{-1}U\not=\emptyset$. By $W\cap (ab)^{-1}W\not=\emptyset$, it follows that $a^{-1}U\cap(ab)^{-1}V\not=\emptyset$ and so $U\cap b^{-1}V\not=\emptyset$. Thus $N_T(U,U)\cap N_T(U,V)\not=\emptyset$.

(3): By Lemma~\ref{lem2.3}, it follows that $(T,X)$ is syndetically transitive. Now for any nonempty open subsets $U_1,U_2, V_1,V_2$ of $X$, $N_T(U_1,V_1)\cap N_T(U_2,V_2)\not=\emptyset$ for $N_T(U_i,V_i)$ is syndetic and thick in $T$. Thus $(T,X)$ is topologically weakly-mixing. (Note that the proof of this implication does not use the assumption that $T$ is abelian.)

(4): The necessity follows immediately from (1). Now to prove the sufficiency, let $(T,X)$ be thickly transitive. Given any nonempty open subsets $U_i, i=1,2,3,4$ of $X$, there are elements $t_1,t_2\in T$ such that $V:=U_1\cap t_1^{-1}U_2\not=\emptyset$ and $W:=t_1^{-1}U_3\cap t_2^{-1}V\not=\emptyset$. Moreover by (2), it follows that there is some $t_3\in T$ such that $W\cap t_3^{-1}W\not=\emptyset$ and $U_4\cap t_3^{-1}W\not=\emptyset$. Letting $t=t_2t_3$, there holds
\begin{gather*}
t_1^{-1}\left(t^{-1}U_2\cap U_3\right)\supseteq t_1^{-1}t^{-1}U_2\cap t^{-1}U_1\cap t_1^{-1}U_3\supseteq t^{-1}V\cap W\supseteq t_3^{-1}W\cap W\not=\emptyset.
\end{gather*}
Thus $t^{-1}U_2\cap U_3\not=\emptyset$. On the other hand,
\begin{gather*}
t^{-1}U_1\cap U_4\supseteq (t_1t)^{-1}U_2\cap t^{-1}U_1\cap U_4=t^{-1}V\cap U_4\supseteq t_3^{-1}W\cap U_4\not=\emptyset.
\end{gather*}
Thus $N_T(U_3, U_2)\cap N_T(U_4,U_1)\not=\emptyset$ and then $(T,X)$ is topologically weakly-mixing.

This proves Corollary~\ref{cor3.3}.
\end{proof}

Notice here that if $T=\{f^n\,|\,n=1,2,\dotsc\}$ and $(T,X)$ is point-transitive, then $T$ is abelian such that every transition $f^n\colon X\rightarrow X$ is a surjection so that $N_T(X,V)=T$. Moreover, for any $(T,X,\pi)$ and nonempty open set $V\subset X$, if $\pi_t\colon X\rightarrow X$ is surjective for every $t\in T$, then $N_T(X,V)=T$. More generally, if $\{t\in T\,|\,\pi_t\textrm{ is surjective}\}$ is thick in $T$, then $N_T(X,V)$ is thick in $T$.

\begin{defn}
Let $(T,X,\pi)$ be a topological semiflow on a compact Hausdorff space $X$. Then:
\begin{itemize}
\item[(1)] A point $x$ is called a \textbf{\textit{distal point}} for $(T,X,\pi)$ if no point in $X$ other than $x$ is proximal to $x$ under $\pi$ (cf.~e.g.,~\cite[Def.~8.2, p.~160]{Fur} and \cite[p.~66]{Aus}).
\item[(2)] A point $x$ is called a \textbf{\textit{locally distal point}} for $(T,X,\pi)$ if no point in $\overline{Tx}$ other than $x$ is proximal to $x$ under $\pi$ (cf.~e.g.,~\cite{V, AF, EEN}).
\end{itemize}
\end{defn}

It is well known that every minimal topologically weakly-mixing map $f$ of a compact metric space $X$ to itself has no distal point for $(X,f)$ (cf.~\cite{V} and also \cite[Theorem~9.12]{Fur}). This will be generalized as follows.

\begin{prop}\label{prop3.5}
If $(T,X,\pi)$ is an \textit{IP}-transitive semiflow on a first countable compact Hausdorff non-singleton space $X$, then there exists no point of $X$ is distal for the semiflow $(T,X,\pi)$.
\end{prop}
\noindent
Note. If $(T,X,\pi)$ is equicontinuous, then `first countable' is superfluous.

\begin{proof}
If otherwise, let $x\in X$ be a distal point of $(T,X,\pi)$. Take two nonempty open subsets $U,V$ of $X$ with $x\in V$ and $U\cap V=\emptyset$. Since $x$ is a distal point of $(T,X,\pi)$, hence $N_T(x,V)$ is an $\textit{IP}^*$-set in $T$ by \cite{DL}. By the \textit{IP}-transitive, it follows that $N_T(U,V)$ contains an \textit{IP}-set in $T$. Thus $N_T(x,V)\cap N_T(U,V)\not=\emptyset$. This together with the first countability implies that there exists some point $y\in U$ such that $x$ is proximal to $y$ under $\pi$. This is a contradiction to the assumption that $x$ is a distal point for $(T,X,\pi)$. Thus the proof is completed.
\end{proof}

Now from Corollary~\ref{cor3.3}-(1) and Proposition~\ref{prop3.5} we can easily obtain the following, which implies that every topologically weakly-mixing transformation $f$ of $X$ does not have any distal points for $(X,f)$.

\begin{cor}\label{cor3.6}
Let $(T,X,\pi)$ be topologically weakly-mixing with a compact Hausdorff non-singleton phase space $X$ which is first countable. If $T$ is abelian such that $N_T(X,V)$ is thick in $T$ for each nonempty open set $V\subset X$, then no point of $X$ is distal for $(T,X,\pi)$.
\end{cor}

Although the topologically weak-mixing is equivalent to the thickly transitive for any continuous surjection of $X$, yet they appear to be conceptually different for the general semigroup setting. It should be noted that the term ``distal'' cannot be replaced by ``locally distal'' in Corollary~\ref{cor3.6}; for example, the full shift system of finite symbols is topologically strongly mixing containing dense periodic points.

Any minimal equicontinuous semiflow with abelian phase semigroup is distal by \cite{DX} and further syndetically transitive by Lemma~\ref{lem2.3} and so it is not \textit{IP}-transitive by Proposition~\ref{prop3.5}. Thus, syndetically transitive $\not\Rightarrow$ \textit{IP}-transitive in general. Particularly, if $X$ is a compact Hausdorff group and $T=X$ and $\pi\colon(t,x)\mapsto tx$ of $T\times X$ to $X$, then $(T,X,\pi)$ is syndetically but not \textit{IP}-transitive.

The following two lemmas are important for our later discussion.

\begin{lem}\label{lem3.7}
If $(T,X,\pi)$ is a topologically weakly-mixing flow (resp. semiflow) on a non-singleton Polish space $X$ with $T$ a group (resp. with $T$ an abelian semigroup), then $T$ is not compact.
\end{lem}

\begin{proof}
Case 1: Let $T$ be a topological group. If $T$ is compact, then $(T,X\times X)$ is minimal. However, this contradicts that $\varDelta_{X\times X}$ is an invariant closed subset of $X\times X$.

Case 2: Let $T$ be an abelian topological semigroup. Suppose $T$ is compact. Then it holds that $T(x_0,y_0)=X\times X$ for any $(x_0,y_0)\in \textit{Trans}(T,X\times X)$. Let $w\in\textit{Trans}(T,X)$ and take $s\in T$ with $s(x_0,y_0)=(w,w)$. From that
\begin{equation*}
s(X\times X)=sT(x_0,y_0)=Ts(x_0,y_0)=T(w,w)=\varDelta_{X\times X},
\end{equation*}
we can see that $X$ is a singleton contradicting that $X$ is a non-singleton Polish space.

The proof of Lemma~\ref{lem3.7} is therefore completed.
\end{proof}

\begin{lem}\label{lem3.8}
Let $(T,X,\pi)$ be a topological flow on a compact Hausdorff space $X$ with phase group $T$ and $x_0\in X$. If $Tx_0=X$ and $\textit{Int}_X(Fx_0)\not=\emptyset$ for some compact subset $F$ of $T$, then $x_0$ is a periodic point.
\end{lem}

\begin{proof}
Let $U$ be any neighborhood of $x_0$ with $U\subset Fx_0$. Since $Tx_0=X$ and $X$ is compact, one can find a finite set $K^{-1}=\{t_0,t_1,\dotsc,t_k\}$ in $T$ such that
$X=\bigcup_{i=0}^kt_iU$.
Therefore, for any $t\in T$, $ktx_0\in U\subset Fx_0$ for some $k\in K$. This shows that $\{t\in T\,|\,tx_0=x_0\}$ is syndetic in $T$.
\end{proof}
\subsection{Multi-dimensional Li-Yorke chaos of weakly-mixing dynamics}\label{sec3.2}
Example~\ref{e1.3} is a completely Li-Yorke chaotic minimal weakly-mixing topological flow with non-abelain phase group $T$. Although a minimal weakly-mixing topological flow with an abelian phase group is never completely Li-Yorke chaotic, yet it is densely Li-Yorke chaotic.
In fact, any minimal weakly-mixing topological flow with abelian phase group $T$ is densely multi-dimensional Li-Yorke chaotic by the following Proposition~\ref{prop3.9}.

We will need a simple observation:
\begin{itemize}
\item If a non-singleton Polish space can support a topologically weakly-mixing semiflow, then this underlying space in question has no isolated points.
\end{itemize}
This fact may follow from the later Lemma~\ref{lem3.20}.

\begin{prop}\label{prop3.9}
Let $(T,X,\pi)$ be a minimal topologically weakly-mixing flow on a non-singleton compact metric space $X$ such that $T$ is an abelian topological group. Then $(T,X,\pi)$ is densely multi-dimensional Li-Yorke chaotic.
\end{prop}

\begin{proof}
Let $k\ge2$ be any given integer. Let $\mathcal{F}=\{F_n\}_1^\infty$ be any given increasing sequence of compact subsets of the abelian topological group $T$.

First by Lemma~\ref{lem3.2}, it follows that $(T,X^k)$ is topologically transitive. Then by Basic Fact~\ref{fac1} in $\S\ref{sec2.1}$, $\textit{Trans}(T,X^k)$ is dense $G_\delta$ in $X^k$. Write
\begin{gather*}
M_k=\{(x_1,\dotsc,x_k)\in\textit{Trans}(T,X^k)\,|\,x_i\not=x_j, 1\le i\not=j\le k\}.
\end{gather*}
Clearly, $M_k$ is also a dense $G_\delta$-subset of $X^k$.

Note that since $(T,X,\pi)$ is weakly mixing, then by Lemma~\ref{lem3.7} it follows that $T$ is not compact and so $F_n\not=T$ for all $n\ge1$.

Let $(x_1,\dotsc,x_k)\in M_k$ be any given. As in (\ref{prop2.14}.3) in the proof of Proposition~\ref{prop2.14}, we can choose a sequence $s_n\in T\setminus F_n, n=1,2,\dotsc$ with $s_n(x_1,\dotsc,x_k)\to(x_1,\dotsc,x_k)$ as $n\to\infty$. This is because otherwise, $T(x_1,\dotsc,x_k)$ is an open subset of $X^k$. Further $Tx_1$ is an open invariant subset of $X$. Since $(T,X,\pi)$ is minimal, hence $Tx_1=X$ and then $Tx=X\ \forall x\in X$. Then by Fort's theorem (cf.~\cite[Theorem~2.13]{Aus}), it follows that $(T,X,\pi)$ is equicontinuous. But since $(T,X,\pi)$ is weakly mixing, it has no non-trivial equicontinuous factor and thus $X$ is a singleton. This is a contradiction.

Finally by Lemma~\ref{lem2.13}, we can choose a dense Cantor set $\Theta$ in $X$ such that for any $k\ge2$ and any distinct $x_1,\dotsc,x_k\in\Theta$, there are $s_n\not\in F_n, t_n\in T$ and some $p\in X$ such that
\begin{gather*}
s_n(x_1,\dotsc,x_k)\to(x_1,\dotsc,x_k),\quad t_n(x_1,\dotsc,x_k)\to(p,\dotsc,p)\quad \textrm{as }n\to\infty.
\end{gather*}
Therefore, $(T,X,\pi)$ is densely multi-dimensional Li-Yorke chaotic relative to any increasing sequence of compact subsets of $T$.

The proof of Proposition~\ref{prop3.9} is thus completed.
\end{proof}

However, for $\mathbb{R}_+$-acting semiflow with no the minimality hypothesis, we can obtain the following, where $\mathbb{R}_+$ is equipped with the usual topology.

\begin{prop}\label{prop3.10}
Let $\pi\colon \mathbb{R}_+\times X\rightarrow X$ be a $C^0$-semiflow on a non-singleton Polish space $X$. If $(X,\mathbb{R}_+,\pi)$ is topologically weakly-mixing, then it is densely multi-dimensional Li-Yorke chaotic.
\end{prop}

\begin{proof}
For any integer $k\ge2$, by Lemma~\ref{lem3.2}, the standard induced $k$-fold product semiflow $\pi\colon \mathbb{R}_+\times X^k\rightarrow X^k$ is topologically transitive, where $X^k=X\times \dotsm\times X$. Then,
\begin{gather*}
P_k=\left\{(x_1,\dotsc,x_k)\in X^k\,|\,\exists t_n\to\infty\textrm{ s.t. }t_n(x_1,\dotsc,x_k)\to(p,\dotsc,p)\in\varDelta_{X^k}\right\}
\end{gather*}
is dense $G_\delta$ in $X^k$, for $X^k$ is a Polish space and $\textit{Trans}(\mathbb{R}_+,X^k)\subseteq P_k$.

Note that given any transitive point $x_0\in X$, $sx_0\not=tx_0$ for any $0\le s<t<\infty$; since otherwise $X$ is homeomorphic to the following space

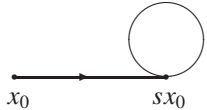
\setlength{\unitlength}{1cm}
\begin{picture}(5,1.7)
\put(0.9,0.15){$x_0$}
\put(1,0.5){\circle*{0.07}}
\put(1,0.5){\vector(1,0){1}}
\put(2,0.5){\line(1,0){1}}
\put(3,0.5){\circle*{0.07}}
\put(2.8,0.15){$sx_0$}
\put(3,1){\circle{1}}
\end{picture}

\noindent
which contradicts the topological transitivity of $\mathbb{R}_+\curvearrowright_\pi X$. Therefore, the topological dimension of any open subset $U$ of the space $X$ is such that $\dim U\ge1$.

Next, let $\mathcal{F}=\{F_n\,|\, n=1,2,\dotsc\}$ be any given increasing sequence of compact subsets of $\mathbb{R}_+$. Let $(x_1,\dotsc,x_k)\in\textit{Trans}(\mathbb{R}_+,X^k)$ be arbitrarily given. We can claim that each $F_n(x_1,\dotsc,x_k)$ contains no nonempty interior; this is because otherwise, some $F_n(x_1,\dotsc,x_k)$ has nonempty interior and then by $F_n(x_1,\dotsc,x_k)\cong F_nx_1$ via the 1-to-1 map $t(x_1,\dotsc,x_k)\mapsto tx_1\ \forall t\in F_n$ it follows that
\begin{equation*}
1\ge\dim F_nx_1=\dim F_n(x_1,\dotsc,x_k)\ge k
\end{equation*}
which is a contradiction to the
weak-mixing. Therefore, $(x_1,\dotsc,x_k)$ is a recurrent point for $(\mathbb{R}_+,X^k)$ in the sense of Definition~\ref{def2.11}. Whence we can choose $s_n\not\in F_n$ with $s_n\to\infty$ such that $s_n(x_1,\dotsc,x_k)\to(x_1,\dotsc,x_k)$.

Let $R_k=X^k\setminus\textit{Trans}(\mathbb{R}_+,X^k)\subset X^k$ for all $k\ge2$. Then by Lemma~\ref{lem2.13}, we can choose a dense multi-dimensional Li-Yorke chaotic set.
\end{proof}

If $s_n\not\in F_n$ is relaxed by $s_n\not\to e$, then we can prove ``weakly-mixing $\Rightarrow$ chaos'' by using Lemma~\ref{lem2.13}. It would be interesting to known whether or not the condition `$\textit{Int}_X(Tx_0)=\emptyset$' is superfluous in the following proposition.

\begin{prop}\label{prop3.11}
Let $(T,X,\pi)$ be a topologically weakly-mixing flow on a non-singleton Polish space $X$ with $T$ a topological group. If there exists some $x_0\in\textit{Trans}(T,X)$ with $\textit{Int}_X(Fx_0)=\emptyset$ for any compact subset $F$ of $T$, then $(T,X,\pi)$ is densely 2D Li-Yorke chaotic.
\end{prop}

\begin{proof}
First from the fact that $(T,X\times X)$ is topologically transitive, it follows that $\textit{Trans}(T,X\times X)$ is dense $G_\delta$ in the Polish space $X\times X$ by Basic Fact~\ref{fac1} in $\S\ref{sec2.1}$.

Let $(x,y)\in\textit{Trans}(T,X\times X)$ be any given; then we can claim $\textit{Int}_{X\times X}(F(x,y))=\emptyset$ for any compact subset $F$ of $T$. For, otherwise, $\textit{Int}_X(Fx)\not=\emptyset$ and so $\textit{Int}_X(Ftx_0)\not=\emptyset$ for some $t\in T$ contradicting the hypothesis.

Then there is a sequence $\{t_n\}$ in $T$ with
$t_n(x,y)\to(p,p)\in\varDelta_{X\times X}$.
On the other hand, for any reference sequence $\mathcal{F}=\{F_n\}_{n=1}^\infty$, by Lemma~\ref{lemII.12}, there is a sequence $\{s_n\}$ in $T$ with $s_n\not\in F_n$ such that
$s_n(x,y)\to(x,y)$.
That is, $(x,y)$ is a 2D Li-Yorke chaotic pair. Then by Lemma~\ref{lem2.13}, one can find a dense subset of $X$ which is a Li-Yorke chaotic set relative to $\mathcal{F}$. This proves Proposition~\ref{prop3.11}.
\end{proof}

Although the weakly-mixing is independent of the topology of the phase semigroup $T$, yet our Li-Yorke chaos and the multi-dimensional chaos both depend heavily on it. Under the discrete topology, many things become relatively simple.

\begin{prop}\label{prop3.12}
Let $(T,X,\pi)$ be a topologically weakly-mixing semiflow on a non-singleton Polish space $X$ with discrete phase semigroup $T$. Then it is densely 2D Li-Yorke chaotic.
\end{prop}

\begin{proof}
We only need to note that every transitive point of $(T,X\times X)$ is recurrent in the sense of Definition~\ref{def2.11} under the discrete topology of $T$. Then the rest arguments are trivial and we thus omit the details.
\end{proof}
If $T$ is abelian, then in view of Lemma~\ref{lem3.2}, similar to Proposition~\ref{prop3.12}, we can obtain the following multi-dimensional chaos. Its $T=\mathbb{Z}_+$ case is \cite[Corollary~3.2-(5)]{AGHSY}.

\begin{prop}\label{prop3.13}
Let $(T,X,\pi)$ be a topologically weakly-mixing semiflow on a non-singleton Polish space $X$ with $T$ an abelian discrete semigroup. Then it is densely multi-dimensional Li-Yorke chaotic.
\end{prop}

If the topology of $T$ is not discrete, the conclusion of Propositions~\ref{prop3.12} and \ref{prop3.13} remains open. We are going to consider topologically weakly-mixing systems with non-discrete abelian phase semigroups in $\S\ref{sec3.3}$.
\subsection{Li-Yorke chaotic sets of weakly-mixing dynamics}\label{sec3.3}
\subsubsection{Proximality}
Let $\mathcal{P}(T,X)$ be the set of all proximal pairs $(x,y)$ of a topological semiflow $(T,X,\pi)$. The following is a generalization of \cite[Theorem~9.12]{Fur} and \cite[Theorem~3.8]{AK} from a cascade on a compact metric space to a general semiflow.

\begin{prop}\label{prop3.14}
If $(T,X,\pi)$ is a thickly transitive topological semiflow on a non-singleton locally compact metric space $X$; then for any almost periodic point $x\in X$, the proximal cell
\begin{gather*}
\mathcal{P}_x(T,X)=\{y\in X\,|\,(x,y)\in\mathcal{P}(T,X)\}
\end{gather*}
is a dense $G_\delta$-set of $X$ and so no point of $X$ is distal for $(T,X,\pi)$.
\end{prop}

\begin{proof}
Let $(X,d)$ be a locally compact metric space, which is not a singleton. We first note that the proximal cell at every point $x\in X$
\begin{gather*}
\mathcal{P}_x(T,X)=\bigcap_{n=1}^\infty\left({\bigcup}_{t\in T}t^{-1}B_{1/n}(tx)\right)
\end{gather*}
is a $G_\delta$-set in $X$, where
$B_{r}(tx)=\{y\in X\,|\,d(tx,y)<r\}$.
So it is enough to show the density of $\mathcal{P}_x(T,X)$. For this, given any almost periodic point $x\in X$ and any nonempty open set $U$ in $X$, we can find $y\in U$ and $t_n\in T$ with $t_nx\to x$ and $t_ny\to x$
(cf.~Proof of \cite[Theorem~9.12]{Fur} for the details). Thus $\mathcal{P}_x(T,X)$ is dense in $X$. This thus concludes Proposition~\ref{prop3.14}.
\end{proof}
\subsubsection{Stable sets}
There is no ``$t_n\to\infty$" in a general semigroup setting. The following lemma makes us get around this hard point, which partially generalizes \cite[Asymptotic Thoerem, p.~262]{HY}.

\begin{lem}\label{lem3.15}
Let $(T,X,\pi)$ be a sensitive topological semiflow on a uniform Baire space $(X,\mathscr{U}_X)$. Then there exists a sensitivity index $\varepsilon\in\mathscr{U}_X$ such that for any sequence $\mathcal{F}=\{F_n\}_{n=1}^\infty$ of compact subsets of $T$
and any $x\in X$, the $(\varepsilon,\mathcal{F})$-stable set of $x$ under the $T$-action
\begin{gather*}
W_\varepsilon^s(x|\mathcal{F}):=\bigcup_{n=1}^\infty W_\varepsilon^s(x|F_n),\quad \textrm{where }W_\varepsilon^s(x|F_n)=\bigcap_{t\not\in F_n} t^{-1}\overline{\varepsilon[tx]},
\end{gather*}
is of the first category in $X$.
\end{lem}

\begin{proof}
Let $\varepsilon^\prime\in\mathscr{U}_X$ be a sensitivity index and let $\varepsilon\in\mathscr{U}_X$ with $\varepsilon^2\subseteq\varepsilon^\prime$.
The sensitivity follows immediately that each $W_\varepsilon^s(x|F_n)$ has no nonempty interior. Indeed, otherwise, one can find some $y\in X$ and $\delta\in\mathscr{U}_X$ so that $\delta[y]\subset W_\varepsilon^s(x|F_n)$. Since $F_n$ is compact, we may let $(ty,tz)\in\varepsilon$ for all $z\in\delta[y]$ and $t\in F_n$ by shrinking $\delta$ if necessary. This implies that for each $z\in\delta[y]$, $(ty,tz)\in\varepsilon^2\subseteq\varepsilon^\prime\ \forall t\in T$. This is impossible.
Thus $W_\varepsilon^s(x|F_n)$ and further $W_\varepsilon^s(x|\mathcal{F})$ is of the first category for $X$ is a Baire space.
\end{proof}

In view of Proposition~\ref{prop2.8}, Lemma~\ref{lem3.15} is applicable for non-minimal M-semiflows. Given any index $\varepsilon\in\mathscr{U}_X$ and any point $x\in X$, for $(T,X,\pi)$ we write
\begin{gather*}
W_\varepsilon^s(T,X;x)=\left\{y\in X\,|\,\exists\textrm{ a compact subset }K\textrm{ of }T\textrm{ s.t. }(tx,ty)\in\varepsilon\ \forall t\in T\setminus K\right\}.
\end{gather*}
The following is a simple consequence of Lemma~\ref{lem3.15}.

\begin{cor}
Let $(T,X,\pi)$ be a sensitive topological semiflow on a compact Hausdorff space with a countable phase semigroup $T$. Then there exists an index $\varepsilon\in\mathscr{U}_X$ such that $W_\varepsilon^s(T,X;x)$ is of the first category for each $x\in X$.
\end{cor}
\subsubsection{Scrambled sets}
The following result shows that every minimal topologically weakly-mixing abelian transformation semigroup has always many Li-Yorke chaotic pairs.

\begin{prop}\label{prop3.17}
Let $(T,X,\pi)$ be a minimal thickly transitive topological semiflow with $T$ abelian on a non-singleton compact metric space $X$. Then there exists a sensitivity constant $\varepsilon>0$ such that
$\mathcal{P}_x(T,X)\setminus W_\varepsilon^s(x|\mathcal{F})$ is a dense $G_\delta$-set in $X$, for any $x\in X$ and any increasing sequence $\mathcal{F}=\{F_n\}_{n=1}^\infty$ of compact subsets of $T$.
\end{prop}

\begin{proof}
Let $x\in X$ be arbitrarily given. By Proposition~\ref{prop3.14}, it follows that the proximal cell $\mathcal{P}_x(T,X)$ is a dense $G_\delta$-set of $X$. On the other hand, by Proposition~\ref{prop3.5} we see that $(T,X,\pi)$ is not distal. Then from Proposition~\ref{prop2.8}-(1), it follows that $(T,X,\pi)$ is sensitive. Let $\varepsilon>0$ be a sensitive constant. Then Lemma~\ref{lem3.15} follows that $W_\varepsilon^s(x|\mathcal{F})$ is of first category in $X$. Therefore, it follows that $\mathcal{P}_x(T,X)\setminus W_\varepsilon^s(x|\mathcal{F})$ is a dense $G_\delta$-set in $X$. This proves Proposition~\ref{prop3.17}.
\end{proof}

Now it is time to state a Li-Yoke chaos result for minimal topologically weakly-mixing semiflows on Polish spaces.

\begin{prop}\label{prop3.18}
For every minimal thickly transitive topological semiflow $(T,X,\pi)$ on a non-singleton compact metric space $X$ with $T$ an abelian semigroup, it is densely Li-Yorke chaotic.
\end{prop}

\begin{proof}
Let $(T,X,\pi)$ be a minimal thickly transitive topological semiflow with $T$ an abelian topological semigroup and where $X$ is a compact metric non-singleton. By Corollary~\ref{cor3.3}-(3) and Lemma~\ref{lem3.7}, $T$ is not compact. Let $\mathcal{F}=\{F_n\}_{n=1}^\infty$ be an increasing sequence of compact subsets of $T$. For any $x\in X$ and $\varepsilon>0$, define
\begin{gather*}
\textrm{LY}_x(T,X)=\left\{y\in X\,|\,\exists\, t_n\in T,\tau_n\in T\setminus F_n\textit{ s.t. }\lim_{n\to\infty}d(t_nx,t_ny)=0,\ \lim_{n\to\infty}d(\tau_nx,\tau_ny)>\varepsilon\right\}
\end{gather*}
By Proposition~\ref{prop3.17}, for some $\varepsilon>0$, $\textrm{LY}_x(T,X)$ is residual in $X$ for any $x\in X$.

Since $(T,X,\pi)$ is sensitive by Proposition~\ref{prop3.14} and Proposition~\ref{prop2.8}, then no point of $X$ is isolated whence $X$ is uncountable.

By induction, we can easily choose a countable dense subset $L$ of $X$ such that every pair of points $x,y\in L, x\not=y$ is a Li-Yorke chaotic pair. Further by using Zorn's lemma, we can find a dense Li-Yorke chaotic set $S\subseteq X$.

This thus proves Proposition~\ref{prop3.18}.
\end{proof}

Similarly, we can obtain the following result which is for every non-abelian group action, which is comparable with Propositions~\ref{prop3.11} and \ref{prop3.12}.

\begin{prop}\label{prop3.19}
Let $(T,X,\pi)$ be a minimal thickly transitive topological flow on a non-singleton compact metric space $X$ with $T$ a topological group. Then $(T,X,\pi)$ is densely Li-Yorke chaotic.
\end{prop}

\begin{proof}
First by Lemma~\ref{lem3.7} and Corollary~\ref{cor3.3}-(3), $T$ is not compact.
By Proposition~\ref{prop3.5}, we see that $(T,X,\pi)$ is minimal but \textit{not distal}. If $(T,X,\pi)$ is sensitive, then the statement holds by an argument same as the proof of Proposition~\ref{prop3.18} above. Thus to prove Proposition~\ref{prop3.19}, it is enough to show $(T,X,\pi)$ is sensitive. For this, it is sufficient to prove that $(T,X,\pi)$ is not equicontinuous by \cite[Theorem~5.7]{KM}. Indeed, if $(T,X,\pi)$ were equicontinuous, then $(T,X,\pi)$ is distal. This is a contradiction. Thus the proof of Proposition~\ref{prop3.19} is completed.
\end{proof}

The following lemma is a simple observation. It may be paraphrased as: ``equicontinuity and weakly-mixing are incompatible.'' See, for example,~\cite{XY, HY} for cascades on compact metric spaces.

\begin{lem}\label{lem3.20}
Let $(T,X,\pi)$ be a topologically weakly-mixing semiflow on a non-singleton Hausdorff uniform space $(X,\mathscr{U}_X)$. Then $(T,X,\pi)$ is sensitive to initial conditions.
\end{lem}

\begin{proof}
By contradiction, suppose that $(T,X,\pi)$ were not sensitive to initial conditions; then there exists a point $x_0\in X$ such that $(T,X,\pi)$ is equicontinuous at $x_0$. Thus the naturally induced product system $(T,X\times X)$ is also equicontinuous at the point $(x_0,x_0)\in X\times X$.

Since $(T,X,\pi)$ is topologically weakly-mixing, then by Def.~\ref{def3.1} it follows that $(T,X\times X)$ is topologically transitive and so $(x_0,x_0)\in\textit{Trans}(T,X\times X)$ by, e.g., \cite[Lemma~4.1]{DX}. In addition, we can take an index $\varepsilon\in\mathscr{U}_X$ such that $\overline{\varepsilon}\subsetneq X\times X$, for $X$ is not a singleton and $X\times X$ is Hausdorff. Therefore, there is some $t_0\in T$ with $t_0(x_0,x_0)\not\in\overline{\varepsilon}$. Since $\pi_{t_0}\colon X\rightarrow X$ is continuous, we can take some $\delta\in\mathscr{U}_X$ such that
$t_0(y,w)\not\in\overline{\varepsilon}$ for all $y,w\in\delta[x_0]$.
This contradicts that $(T,X,\pi)$ is equicontinuous at the point $x_0$.

The proof of Lemma~\ref{lem3.20} is thus completed.
\end{proof}

Finally, we can readily prove our final two main results on Li-Yorke chaos, which are comparable with Proposition~\ref{prop3.9}.

\begin{prop}\label{prop3.21}
Let $(T,X,\pi)$ be a topologically weakly-mixing flow on a non-singleton Polish space $X$ with $T$ a topological group. Then it is densely Li-Yorke chaotic.
\end{prop}

\begin{proof}
If $Int_{X\times X}T(x,y)=\emptyset$ for any $(x,y)\in\textit{Trans}(T,X\times X)$ with respect to $(T,X\times X)$, then each $(x,y)\in\textit{Trans}(T,X\times X)$ is recurrent in the sense of Definition~\ref{def2.11} for $(T,X\times X)$. We can then conclude easily the statement by Lemma~\ref{lem2.13}.

Let $Int_{X\times X}T(x,y)\not=\emptyset$ for some $(x,y)\in\textit{Trans}(T,X\times X)$; and then $T(x,y)$ is open for every $(x,y)\in\textit{Trans}(T,X\times X)$. Moreover, $T(x,y)=T(x^\prime,y^\prime)$ for any $(x,y), (x^\prime,y^\prime)\in\textit{Trans}(T,X\times X)$. Then there are two cases: (a) $T(x_0,y_0)\subsetneq X\times X\setminus\varDelta_{X\times X}$ for any $(x_0,y_0)\in\textit{Trans}(T,X\times X)$, and (b) $T(x_0,y_0)=X\times X\setminus\varDelta_{X\times X}$ for any $(x_0,y_0)\in\textit{Trans}(T,X\times X)$.

(a) Let $(x_0,y_0)\in\textit{Trans}(T,X\times X)$ be any given and let $\{F_n\}_{n=1}^\infty$ be a sequence of compact subsets of $T$.
Clearly there is a sequence $s_n\in T\setminus F_n, n=1,2,\dotsc$ such that
\begin{equation*}
\liminf_{n\to\infty}d(s_nx_0,s_ny_0)>0.
\end{equation*}
Thus $(x_0,y_0)$ is a Li-Yorke chaotic pair for $(T,X,\pi)$. Since $\textit{Trans}(T,X\times X)$ is a dense $G_\delta$-set in $X\times X$ and $X$ has no isolated points, hence by Lemma~\ref{lem2.13} we can conclude the statement in this case.

(b) Let $T(x,y)=X\times X\setminus\varDelta_{X\times X}$ for any $(x,y)\in\textit{Trans}(T,X\times X)$. Then $T(x,y)=X\times X\setminus\varDelta_{X\times X}$ for any $(x,y)\in X\times X\setminus\varDelta_{X\times X}$, for $T$ is a group. This implies that for any $x\in X$, the proximal cell $\mathcal{P}_x(T,X)=X$.

Let $\mathcal{F}=\{F_n\}_{n=1}^\infty$ be an increasing sequence of compact subsets of $T$. For any $x\in X$ and $\varepsilon>0$, define
\begin{gather*}
\textrm{LY}_x=\left\{y\in X\,|\,\exists\, t_n\in T,s_n\in T\setminus F_n\textit{ s.t. }\lim_{n\to\infty}d(t_nx,t_ny)=0,\ \lim_{n\to\infty}d(s_nx,s_ny)>\varepsilon\right\}
\end{gather*}
By Lemmas~\ref{lem3.20} and~\ref{lem3.15}, it follows that for some $\varepsilon>0$, $\textrm{LY}_x$ is residual in $X$ for any $x\in X$.
Finally, we can find a densely Li-Yorke chaotic set for $(T,X)$ similar to the proof of Proposition~\ref{prop3.18}.

This thus proves Proposition~\ref{prop3.21}.
\end{proof}

The special case of this proposition that $T$ is a countable discrete group is \cite[Theorem~1.2]{WZ} under the guise of ``weakly mixing subsets'' by different approaches.

If $T$ is not a group but only a semigroup, then we can obtain the dense Li-Yorke chaos under the abelian condition by using Furstenberg's intersection lemma.

\begin{prop}\label{prop3.22}
Let $(T,X,\pi)$ be a topologically weakly-mixing semiflow on a non-singleton compact metric space $X$ with $T$ an abelian topological semigroup. Then $(T,X,\pi)$ is densely Li-Yorke chaotic.
\end{prop}

\begin{proof}
By Furstenberg's intersection lemma (Lemma~\ref{lem3.2}) and the finite intersection property of compact space, it follows that
\begin{gather*}
{\bigcap}_{F\in\mathcal{F}(T,X,\pi)}\overline{Fx}\not=\emptyset\quad \forall x\in X.
\end{gather*}
This implies that for any $x\in X, z\in\bigcap_{F\in\mathcal{F}(T,X,\pi)}\overline{Fx}$ and each nonempty open subset $U$ of $X$, by induction one can find some point $y\in U$ and some sequence $t_n\in T$ such that $t_nx\to z$ and $t_ny\to z$. Thus $\mathcal{P}_x(T,X)$ is a dense $G_\delta$-set in $X$, for every $x\in X$; cf.~\cite[Theorem~3.8]{AK} for cascades.

On the other hand, any topologically weakly-mixing semiflow is clearly sensitive to initial conditions by Lemma~\ref{lem3.20}. Moreover, $T$ is not compact by Lemma~\ref{lem3.7}. Whence $\textrm{LY}_x(T,X)$ relative to any given increasing sequence $\mathcal{F}$ of compact subsets of $T$, which is defined as in the proof of Proposition~\ref{prop3.18}, is a dense $G_\delta$-set in $X$ by Lemma~\ref{lem3.15}. Moreover, $X$ contains no isolated points. Further, by induction and Zorn's lemma, we can easily choose a dense Li-Yorke chaotic set for $(T,X,\pi)$.

This thus concludes the proof of Proposition~\ref{prop3.22}.
\end{proof}

It is easy to see that Propositions~\ref{prop3.19} and \ref{prop3.18} are simple consequences of Propositions~\ref{prop3.21} and \ref{prop3.22} by different approaches, respectively.
\subsection{Li-Yorke sensitivity}
In the following definition that is stronger than sensitivity to initial conditions, let $(T,X,\pi)$ be a topological semiflow on a uniform space $(X,\mathscr{U}_X)$ with phase semigroup $T$.

\begin{defn}
$(T,X,\pi)$ is called \textbf{\textit{Li-Yorke sensitive}} if there exists an index $\varepsilon\in\mathscr{U}_X$ such that every $x\in X$ is a limit of points in $\mathcal{P}_x(T,X)\setminus W_\varepsilon^s(X,T;x)$; that is, $x\in\overline{\mathcal{P}_x(T,X)\setminus W_\varepsilon^s(T,X;x)}$. See \cite[Def.~3.5]{AK} for cascades on compact metric spaces.
\end{defn}

Then from the proof of Proposition~\ref{prop3.22} we can easily obtain the following, which generalizes \cite[Corollary~3.9]{AK}.

\begin{prop}
Let $(T,X,\pi)$ be a topologically weakly-mixing semiflow on a non-singleton compact metric space $X$ with $T$ an abelian topological semigroup. Then $(T,X,\pi)$ is Li-Yorke sensitive.
\end{prop}

\subsection{Li-Yorke chaos of strongly-mixing flows}\label{sec3.5}
Let $T$ be a topological group. A flow $T\curvearrowright_\pi X$ on a topological space $X$ with the property that for any nonempty open subsets $U,V$ of $X$ $N_T(U,V)$ is nonempty co-compact in $T$ is known as \textbf{\textit{topologically strongly-mixing}}.

\begin{cor}\label{cor3.25}
Let $(T,X,\pi)$ be a topologically strongly-mixing flow on a non-singleton Polish space $X$ with $T$ a topological group. Then $(T,X,\pi)$ is densely Li-Yorke chaotic.
\end{cor}

\begin{proof}
This follows from Proposition~\ref{prop3.21}, for topologically strongly-mixing implies topologically weakly-mixing.
\end{proof}

This corollary generalizes \cite{XY} which is for $\mathbb{Z}$-actions on compact metric spaces with no isolated points.
\section*{\textbf{Acknowledgments}}%
The authors thank Prof.~Ethan~Akin, Prof.~Xiangdong Ye and Prof.~Guohua Zhang for sending us their articles \cite{AAB,AK}, \cite{AGHSY} and \cite{WZ}, respectively. Thanks are also due to Prof. Michael~Megrelishvili and Dr. Zhijing Chen for their many comments on the proof details.

This work was partly supported by National Natural Science Foundation of China (Grant Nos. 11431012, 11271183) and PAPD of Jiangsu Higher Education Institutions.

\end{document}